%% file: main_krylov.tex
\documentclass[11pt,letterpaper]{article}

\usepackage[utf8]{inputenc}
\usepackage{comment}
\usepackage{microtype}
\usepackage{graphicx}
\usepackage{subfig}
\usepackage{booktabs} %
\usepackage{scalerel}
\usepackage[margin=1in]{geometry}
\usepackage[style=authoryear-comp, giveninits=true, maxbibnames=15, maxcitenames=2, natbib=true,
  maxalphanames=10, backend=biber, sorting=nty, backref=true, uniquename=false]{biblatex}
\DefineBibliographyStrings{english}{backrefpage = {page},backrefpages = {pages},}
\DeclareNameAlias{sortname}{given-family}
\renewbibmacro{in:}{%
  \ifentrytype{article}{}{\printtext{\bibstring{in}\intitlepunct}}}

\usepackage{amssymb}
\usepackage{amsthm}
\usepackage{amsmath}
\usepackage{mathtools}

\usepackage{nicefrac}

\usepackage{algorithm}
\usepackage{algpseudocode}
\usepackage[shortlabels]{enumitem}

\usepackage{url}
\usepackage{float}

\usepackage{pifont}
\usepackage{hhline}
\usepackage{multirow}
\usepackage{bm}

\input{math_commands.tex}

\newcommand{\reals}{\mathbb{R}}

\newcommand{\calO}{\mathcal{O}}
\newcommand{\calV}{\mathcal{V}}
\newcommand{\calK}{\mathcal{K}}

\newtheorem{assumption}{Assumption}

\newtheorem{theorem}{Theorem}
\newtheorem{lemma}[theorem]{Lemma}
\newtheorem{proposition}[theorem]{Proposition}

\newtheorem{remark}{Remark}
\numberwithin{assumption}{section}
\numberwithin{definition}{section}
\numberwithin{theorem}{section}
\numberwithin{remark}{section}

\usepackage{hyperref}
\hypersetup{colorlinks,citecolor=blue,linktocpage,breaklinks=true}
\begin{document}

\title{Krylov Cubic Regularized Newton: {A Subspace Second-Order Method with Dimension-Free Convergence Rate}\let\thefootnote\relax\footnotetext{The work of RJ was partially done while interning at Amazon Web Services. SS, MH, and VC hold concurrent appointments as an Amazon Scholar and as a faculty at Technion, University of Minnesota, and EPFL, respectively. This paper describes their work performed at Amazon.}}

\author{Ruichen Jiang\thanks{Department of Electrical and Computer Engineering, The University of Texas at Austin\ \{rjiang@utexas.edu, mokhtari@austin.utexas.edu\}} 
\and 
        Parameswaran Raman\thanks{Amazon Web Services \ \{prraman@amazon.com\}}%
\and 
        Shoham Sabach\thanks{Faculty of Data and Decision Sciences, Technion -- Israel Institute of Technology \ \{ssabach@technion.ac.il\}}
\and
        Aryan Mokhtari$^*$
\and 
        Mingyi Hong\thanks{Department of Electrical and Computer Engineering, University of Minnesota \ \{mhong@umn.edu\}}
\and    
        Volkan Cevher\thanks{LIONS, IEM, STI, Ecole Polytechnique Fédérale de Lausanne \ \{volkan.cevher@epfl.ch\}}
}

\date{}

\maketitle

\begin{abstract}  
Second-order optimization methods, such as cubic regularized Newton methods, are known for their rapid convergence rates; nevertheless, they become impractical in high-dimensional problems due to their substantial memory requirements and computational costs. One promising approach is to execute second-order updates within a lower-dimensional subspace, giving rise to \textit{subspace second-order} methods. However, the majority of existing subspace second-order methods randomly select subspaces, consequently resulting in slower convergence rates depending on the problem's dimension  $d$. 
In this paper, we introduce a novel subspace cubic regularized Newton method that achieves a dimension-independent global convergence rate of $\bigO\left(\frac{1}{mk}+\frac{1}{k^2}\right)$ for solving convex optimization problems. Here, $m$ represents the subspace dimension, which can be significantly smaller than $d$. Instead of adopting a random subspace, our primary innovation involves performing the cubic regularized Newton update within the \emph{Krylov subspace} associated with the Hessian and the gradient of the objective function. This result marks the first instance of a dimension-independent convergence rate for a subspace second-order method. Furthermore, when specific spectral conditions of the Hessian are met, our method recovers the convergence rate of a full-dimensional cubic regularized Newton method.   Numerical experiments show our method converges faster than existing random subspace methods, especially for high-dimensional problems.
\end{abstract}

\newpage
\newpage

\section{Introduction}

In this paper, we consider the following unconstrained minimization problem 
\begin{equation*}
  \min_{\vx \in \mathbb{R}^d} f(\vx),
\end{equation*}
where $f: \reals^d \rightarrow \reals$ is convex and twice continuously  differentiable.
We focus on the use of second-order methods for solving this problem, particularly the Cubic Regularized Newton (CRN) method \citep{griewank1981modification,nesterov2006cubic}. %
The CRN method stands out for its fast convergence rate. Specifically, when dealing with convex functions, it exhibits a global convergence rate of $\mathcal{O}(1/k^2)$, with $k$ denoting the number of iterations. Furthermore, in cases where $f$ is strongly convex, it attains a superlinear convergence rate~{\citep{nesterov2006cubic}}. Nevertheless, the main drawback associated with the CRN method is its substantial computational cost per iteration, particularly when the problem's dimensionality $d$ is high, leading to unfavorable scaling.
 This arises from the necessity to solve a cubic subproblem at each iteration, demanding a minimum of $\calO(d^3)$ arithmetic operations. As a result, CRN becomes impractical for optimization problems with high dimensions, a common scenario in modern machine learning applications.

{To reduce the computational cost, \citet{Hanzely2020} proposed the stochastic subspace cubic Newton (SSCN) method, which can be regarded as a subspace variant of CRN.   
Inspired by first-order coordinate descent methods~\citep{luo1992convergence,nesterov2012efficiency,wright2015coordinate}, their key proposition is to solve the cubic subproblem over a random low-dimensional subspace of dimension $m \ll d$, instead of the full-dimensional space $\mathbb{R}^d$. As a result, this strategy effectively reduces the dimension of the cubic subproblem to $m$, and thus it can be solved efficiently using $\bigO(m^3)$ arithmetic operations. However, this efficiency comes at a cost: SSCN suffers from a slower convergence rate of $\bigO\left(\frac{d-m}{m}\cdot\frac{1}{k}+ \left(\frac{d}{m}\right)^2 \cdot \frac{1}{k^2}\right)$ that scales with the problem's dimension $d$. Additionally, formulating the low-dimensional cubic subproblem requires computing the gradient and the Hessian of the objective over the chosen subspace, which also needs to be taken into consideration. 
In certain special cases, such as generalized linear models, they can be computed with a complexity of $\calO(m)$ and $\calO(m^2)$, respectively, as demonstrated by \citep{Hanzely2020}. In this case, the dominant cost comes from solving the cubic subproblem, leading to an overall complexity of $\calO(m^3)$. In Table~\ref{tab:second_order}, we report the per-iteration cost of SSCN for such favorable cases.
In general, however, the cost of computing the subspace gradient and the subspace Hessian could be dependent on the problem's dimension~$d$. {For instance,  \citet{gower2019rsn} proposed computing the subspace Hessian via $m$ back-propagation passes. Since each back-propagation requires $\calO(d)$ arithmetic operations, in this case the per-iteration cost could be $\bigO(md)$.  }%

Given the discussions above, we are motivated by the following question: ``\textit{Can we improve the dimensional dependence of subspace second-order methods?}''  Intuitively, this problem stems from the fact that the subspace is chosen uniformly random at each iteration, oblivious to the objective function we aim to optimize. As a result, such a random subspace is unlikely to contain a ``good'' descent direction of the objective, hindering the convergence of the subspace method. 
In this paper, we argue that employing a subspace customized to the local geometry of the objective yields a convergence rate that is independent of the dimensionality.
Specifically, we propose the Krylov CRN method, where we perform the CRN update over the \emph{Krylov subspace} associated with the Hessian and the gradient of the objective function. 
Our contributions are summarized as follows: 
\begin{itemize}
    \item In the convex case, we prove a dimension-free convergence rate of $\calO(\frac{1}{mk}+\frac{1}{k^2})$ for our proposed method, where $m$ is the subspace dimension and $k$ is the number of iterations. 
    In particular, we shave a factor of $\calO(d)$ from the iteration complexity of SSCN when $m \ll d$ (see Table~\ref{tab:second_order}). {Additionally, in the strongly convex case, our proposed method achieves a linear rate of convergence, and we again shave a $\calO(d)$ factor from the iteration complexity of SSCN when $m \ll d$.}
    Furthermore, we show that our method can be implemented using the Lanczos method, requiring one gradient evaluation and $m$ Hessian-vector products per iteration, and the resulting cubic subproblem can be solved using $\calO(m)$ arithmetic operations. {{Hence, in the worst case, the per-iteration cost of our method is $\calO(md)$, resulting from the computation of $m$ Hessian-vector products.}}

    \item We show that our method is capable of exploiting the spectral structure of the objective's Hessian. Specifically, we characterize the convergence rate of our method in terms of a spectral quantity of the objective's Hessian (cf.~Theorem~\ref{thm:main}), which could lead to faster convergence rates when the Hessian possesses some structure. 
    For instance, if the Hessian has at most $m$ distinct eigenvalues, then our method recovers the $\bigO(1/k^{2})$ rate of the full-dimensional CRN method.

    \item To demonstrate the efficacy of our algorithm, we perform numerical experiments on high-dimensional logistic regression problems, as exemplified by Figure~\ref{fig:intro}. Specifically, we observe from Figure~\subref{fig:intro_a} that our proposed Krylov CRN method {makes much more progress than SSCN in each iteration and closely follows the loss curve of the full-dimensional CRN, even with a modest subspace dimension of $m=10$. Moreover, given the same amount of computation time, Figure~\subref{fig:intro_b} shows that} our proposed method can converge much faster than both SSCN and the full-dimensional CRN.   
\end{itemize}

\begin{table}[t!]\scriptsize
  \renewcommand{\arraystretch}{1.0}
  \centering
  \caption{The comparison of CRN, SSCN, and our method in terms of per-iteration cost and convergence rate.  $^*$We assume that 
  the cost of computing the subspace gradient and the subspace Hessian is $\calO(m)$ and $\calO(m^2)$, respectively. $^{**}$We assume that the cost of Hessian-vector product evaluations is $\calO(d)$.} \vspace{1mm}
  \label{tab:second_order}
   {%
      \begin{tabular}{cccc}
          \toprule
{\qquad Methods}   & Per-iteration cost    & Convergence rate    \\ \midrule 
{CRN \citep{nesterov2006cubic}} & 
$\calO(d^3)$
& $\calO(\frac{1}{k^2})$  \\ \midrule
SSCN \citep{Hanzely2020} &   {$\calO( m^3)$$^*$}    & $\calO(\frac{d-m}{m}\cdot\frac{1}{k}+\frac{d^2}{m^2}\cdot\frac{1}{k^2})$ \\ \midrule
\textbf{Krylov CRN ({ours})} & $\calO(md)^{**}$
& $\calO(\frac{1}{mk} + \frac{1}{k^2})$  \\ \bottomrule
      \end{tabular}%
  }%
\end{table}

\noindent\textbf{Additional related work.} %
The idea of using the Krylov subspace method for solving the subproblem in CRN has been previously explored by \citet{cartis2011adaptive} and \citet{carmon2018analysis}. However, there are two key distinctions between their work and ours. Firstly, they focus on solving nonconvex optimization problems and establish results for finding approximate stationary points, whereas our focus is on the convex and strongly convex settings. Secondly, and more importantly, their analysis requires the Krylov subspace solution to be an approximate minimizer of the full-dimensional cubic subproblem, and hence the subspace dimension $m$ needs to depend on the final target accuracy $\epsilon$. 
{{Consequently, to attain a final accuracy of $\epsilon$, they require the subspace dimension to be $m = \tilde{\calO}(\epsilon^{-\frac{1}{4}})$ \citep{Carmon2020}.}}
In contrast, our results hold for any constant value of $m$, which can be independent of the iteration index $k$ or the target accuracy $\epsilon$.

In addition to CRN, another classical second-order method is the damped Newton's method, and its subspace variants have also been considered by \citet{gower2019rsn,hanzely2023sketch}. Since their convergence rates are shown under a different set of assumptions, their results are not directly comparable with ours. In addition, we note that their convergence rates suffer from the same issue of dimensional dependence as they also rely on random subspaces. 

\section{Preliminaries}\label{sec:prelim}
\vspace{-.2em}
\begin{figure}[!t]
\vspace{-1em}
    \centering
    \subfloat[]
    {\includegraphics[width=0.45\linewidth]{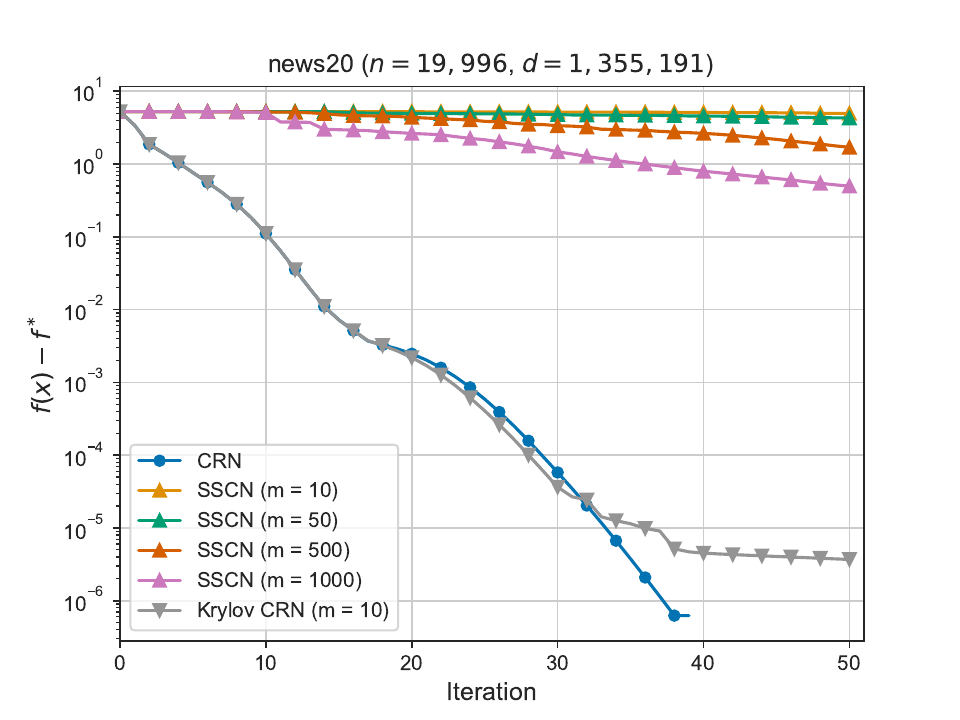}\label{fig:intro_a}}
    \subfloat[]
    {\includegraphics[width=0.45\linewidth]{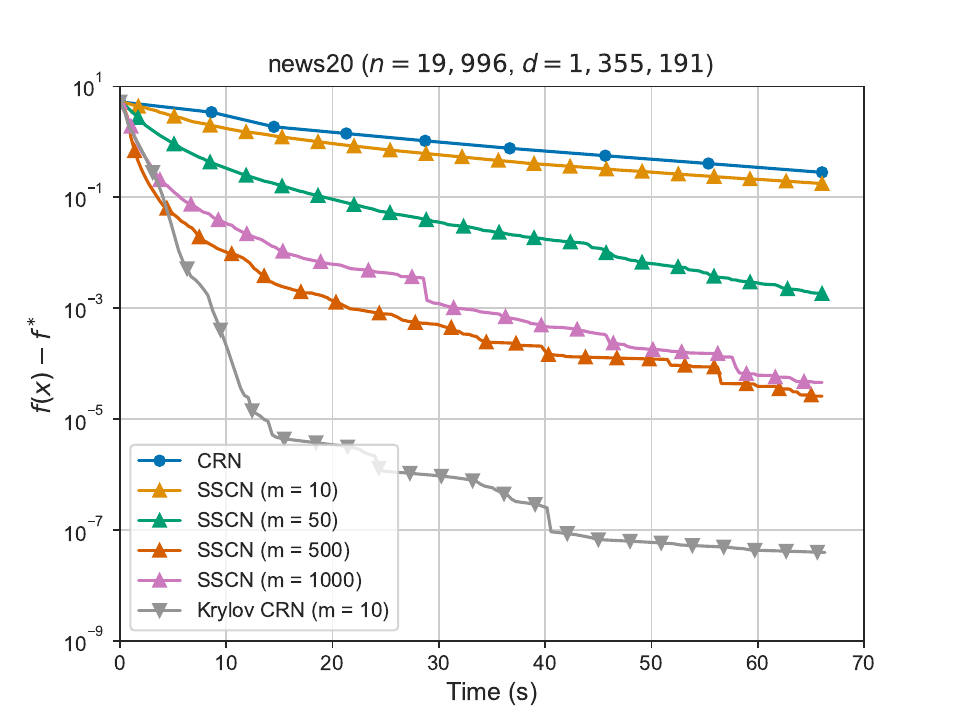}\label{fig:intro_b}}
    \caption{Numerical comparisons of CRN, SSCN, and our proposed Krylov CRN method on a logistic regression problem. Here, $m$ denotes the dimension of the subspace, $n$ is the number of samples, and $d$ is the number of parameters. See Section~\ref{sec:numerical} for full details.}
    \label{fig:intro}
    \vspace{-.3em}
\end{figure}

In this section, we first introduce the necessary assumptions and highlight the auxiliary results arising from these conditions. Then, in Section~\ref{subsec:subspace_newton}, we discuss the original cubic regularized Newton (CRN) method as well as its subspace variants studied in the literature.
Note that our required assumptions are all common in the analysis of CRN-type methods~\citep{nesterov2006cubic}. 
\begin{assumption}\label{assum:convex}
  The function $f: \reals^d \rightarrow \reals$ is convex. %
\end{assumption}

\begin{assumption}\label{assum:bounded_sublevel}
    The function $f$ is bounded from below and has bounded level-sets. %

\end{assumption}

\begin{assumption}\label{assum:hessian}
  The Hessian $\nabla^2 f$ is $L_2$-Lipschitz: $\|\nabla^2 f(\vx) - \nabla^2 f(\vy) \| \leq L_2 \|\vx-\vy\|$, for all $\vx,\vy \in \mathbb{R}^{d}$. 
\end{assumption}

Utilizing conventional arguments (see, e.g., \cite[Lemma 1.2.4]{Nesterov2018}), one key implication of  Assumption~\ref{assum:hessian} is that we can control the difference between the function $f(\vx)$ and its quadratic approximation, as demonstrated in the following proposition.

\begin{proposition}\label{prop:taylor}
  If Assumption \ref{assum:hessian} holds, then for any $\vx,\vx' \in \mathbb{R}^d$, we have $$\left|f(\vx)-f(\vx')-\nabla f(\vx')^\top (\vx-\vx')-\frac{1}{2} (\vx-\vx')^\top \nabla^2 f(\vx')(\vx-\vx')\right| \leq \frac{L_2}{6}\|\vx-\vx'\|^3.$$ 
\end{proposition}

\subsection{Subspace Cubic Regularized Newton}\label{subsec:subspace_newton}
In this part, we start by providing an overview of the classic full-dimensional CRN method \citep{nesterov2006cubic} as well as its subspace variant. 
Recall that the CRN method is motivated by the following simple observation:  using Proposition~\ref{prop:taylor}, we can upper bound the function $f$ by its quadratic approximation with a cubic regularizer. Specifically, 
let $\vg_k$ and $\mH_k$ denote the vector $\nabla f(\vx_k)$ and the matrix $\nabla^2 f(\vx_k)$, respectively; we will use these notations throughout the paper. 
Then, given a regularization parameter $M \geq L_2$ and the current iterate $\vx_k$,  %
we have 
\begin{equation}
  f(\vx) \leq f(\vx_k) + \vg_k^\top (\vx-\vx_k) + \frac{1}{2}(\vx-\vx_k)^\top \mH_k (\vx-\vx_k) + \frac{M}{6}\|\vx-\vx_k\|^3. \label{eq:cubic_upper_bnd} 
\end{equation}
Thus, a reasonable choice is to select the new iterate $\vx_{k+1}$ as the minimizer of the upper bound of $f(\vx)$ that is given in the right-hand side of \eqref{eq:cubic_upper_bnd}. 
Specifically, starting with any $\vx_{0} \in \mathbb{R}^{d}$, the update rule of CRN can be written as 
\begin{align}
    \vs_{k} & := \argmin_{\vs \in \mathbb{R}^d} \Bigl\{ \vg_k^\top \vs + \frac{1}{2}\vs^\top \mH_k \vs + \frac{M}{6}\|\vs\|^3 \Bigr\} \label{eq:cubic_newton} \\ 
    \vx_{k+1} & = \vx_k + \vs_k. \nonumber  
\end{align}  
It is well-known that the CRN method achieves a fast convergence rate  
 of $\calO(1/k^2)$ if Assumptions~\ref{assum:convex}-\ref{assum:hessian} hold \citep{nesterov2006cubic}. However, solving the cubic subproblem in \eqref{eq:cubic_newton} could be computationally prohibitive, particularly in a high-dimensional setting where $d$ is large. To better elaborate on this issue, note that the standard approach for solving the above subproblem (as discussed in \citep{conn2000trust,cartis2011adaptive}) is to reformulate \eqref{eq:cubic_newton} as a one-dimensional nonlinear equation in $\lambda$, which is given by 
$
\lambda = \frac{M}{2}\|(\mH_k + \lambda \mI)^{-1} \vg_k\|
$.
By applying Newton's method to find its root, to be denoted by $\lambda^*$, the solution of \eqref{eq:cubic_newton} can then be determined as $\vs_k = -(\mH_k + \lambda^* \mI)^{-1} \vg_k$. Therefore, finding $\vs_k$ at each iteration demands %
solving multiple linear systems of equations of dimension $d$, which requires $\mathcal{O}(d^3)$ arithmetic operations. 
Hence, the cost of finding the solution of \eqref{eq:cubic_newton} in CRN scales poorly with the dimension $d$.

To mitigate this issue, subspace variants of the CRN method have been proposed in the literature by  \citet{doikov2018randomized} and \citet{Hanzely2020}.
Their main idea is to restrict the variable $\vs$ in the subproblem in \eqref{eq:cubic_newton} to a low-dimensional subspace $\mathcal{V}_k$, thus reducing the effective dimension of the subproblem. 
Concretely, let $\mV_k \in \mathbb{R}^{d\times m} $ be a tall matrix whose columns form an orthogonal basis for $\mathcal{V}_k$, where $m$ is the subspace dimension. Then, solving the cubic subproblem in \eqref{eq:cubic_newton} over the subspace $\mathcal{V}_k$ is equivalent to
\begin{align}
  \vz_{k} &:= \argmin_{\vz \in \mathbb{R}^m } \Bigl\{ \tilde{\vg}_k^\top \vz  + \frac{1}{2}\vz^\top \tilde{\mH}_k \vz + \frac{M}{6}\|\vz\|^3 \Bigr\}, \label{eq:general_subspace}\\
  \vs_k &:= \mV_k \vz_k, \nonumber
\end{align} 
where $\tilde{\vg}_k := \mV_k^\top \vg_k \in \mathbb{R}^m$ and $\tilde{\mH}_k := \mV_k^\top \mH_k \mV_k \in \mathbb{R}^{m\times m}$ are the gradient and Hessian in the subspace, respectively. Note that the effective dimension of the subproblem in \eqref{eq:general_subspace} is $m$  and hence it can be solved in $O(m^3)$ time. %
In particular, \citet{doikov2018randomized} considered objective functions with a block separable structure, and their method defines the subspace $\mathcal{V}_k$ and its corresponding projection matrix $\mV_k$ by sampling a random block of coordinates at each iteration.
Later, \citet{Hanzely2020} extended this idea of the subspace CRN method to general objective functions and general random subspaces satisfying the assumption $\E[\mV_k \mV_k^\top ]= \frac{m}{d}\mI$ for all $k\geq 0$.

Although the proposed randomized subspace CRN methods in \citep{doikov2018randomized, Hanzely2020} successfully reduce the cost of solving the subproblem in the CRN method, their convergence rate depends on the problem's dimension $d$, as highlighted in Table~\ref{tab:second_order}. 
{Moreover, in addition to Assumptions~\ref{assum:convex}-\ref{assum:hessian}, they also require an extra assumption that the objective function Hessian has bounded eigenvalues, i.e.,  $\nabla^2 f(\vx) \preceq L_1 \mI$ for all $\vx\in \reals^d$.
Under these assumptions, SSCN is shown to achieve a rate of $\calO(\frac{d-m}{m}\cdot\frac{L_1}{k}+\frac{d^2}{m^2}\cdot\frac{1}{k^2})$.}
{This issue raises the question of whether there is a better choice of the subspace so that the convergence rate of the subspace CRN method is independent of the dimension and has a better dependence on $L_1$}, while ensuring that the cost of solving the subproblem remains affordable. 
In the next section, we show that our proposed Krylov subspace CRN method can achieve this goal.

\section{The Proposed Algorithm}
In this section, we first explain the rationale behind the use of the Krylov subspace before introducing our proposed Krylov CRN method.

\subsection{Rationale behind the Krylov Subspace}\label{subsec:rationale}
To motivate the use of the Krylov subspace, let us examine how introducing the subspace $\mathcal{V}_k$ would change the convergence analysis and under what conditions on $\calV_k$ the subspace method can still retain the fast convergence rate of $\calO(1/k^{2})$. 
In the analysis of the CRN method in \citep{nesterov2006cubic}, given the current iterate $\vx_{k}$ and the next iterate $\vx_{k + 1}$ which is computed based on \eqref{eq:cubic_newton}, we use the following crucial inequality
\begin{equation}\label{eq:key_inequality}
  f(\vx_{k+1}) \leq f(\vx_k) +  \vg_k^\top \vs  + \frac{1}{2}\vs^\top \mH_k \vs + \frac{M}{6}\|\vs\|^3,
\end{equation}
which is true for any $M \geq L_2$ and for all $\vs\in \mathbb{R}^d$. Importantly, \eqref{eq:key_inequality} provides an upper bound on $f(\vx_{k+1})$ for \emph{any} $\vs\in \mathbb{R}^d$, which gives us the freedom to choose a well-designed $\vs$ related to the optimal solution $\vx^*$ in the convergence analysis. 

Given that the inequality in \eqref{eq:key_inequality} plays a significant role in demonstrating the $\bigO(1/k^2)$ convergence rate of CRN, we need to investigate the equivalent inequality that can be established for the subspace variant of the CRN method. This inequality may assist us in making well-informed choices for the subspace $\mathcal{V}_k$ to minimize its impact on the convergence rate. More precisely, if we follow the update in~\eqref{eq:general_subspace}, then we can show that 
for any $\vs\in \mathbb{R}^d$, we have 
\begin{equation}
  f(\vx_{k+1}) \leq f(\vx_k) +\vg_k^\top \mP_k \vs  + \frac{1}{2}(\mP_k\vs)^\top \mH_k \mP_k\vs + \frac{M}{6}\|\mP_k\vs\|^3, \label{eq:key_inequality_subspace}
\end{equation}
where $\mP_k := \mV_k \mV_k^\top$ is the orthogonal projection matrix associated to the subspace $\mathcal{V}_k$ (for more details check the proof of  Lemma~\ref{lem:one_step_krylov} in the Appendix}). By comparing \eqref{eq:key_inequality_subspace} and \eqref{eq:key_inequality} term by term, we observe that \eqref{eq:key_inequality} still holds true if  $\mP_k$ is selected such that the following conditions hold for all $\vs\in \mathbb{R}^d$:
\begin{enumerate}[(A), leftmargin=1cm]
  \item $ \vg_k^\top \mP_k \vs \leq  \vg_k^\top \vs $.
  \item $ (\mP_k\vs)^\top \mH_k \mP_k\vs \leq \vs^\top \mH_k \vs$.
  \item $\|\mP_k \vs\| \leq \|\vs\|$.
\end{enumerate} 
Note that Condition (C) holds automatically since $\mP_k$ is an orthogonal projection matrix. On the other hand, the first two conditions impose restrictions on the matrix $\mP_k$, which in turn constrains the subspace $\mathcal{V}_k$ that we can choose. Below, we will outline the conditions under which the requirements in (A) and (B) would be met. 
\begin{proposition}\label{prop:maximal_krylov}
  Let $\mP_k$ be the orthogonal projection matrix associated to the subspace $\mathcal{V}_k$. 
  \begin{enumerate}[(i), leftmargin=1cm]
\vspace{-2mm}
    \item Condition (A) holds true if and only if $\vg_k \in \mathcal{V}_k$. 
    \item Condition (B) holds true if and only if $\mathcal{V}_k$ is an invariant subspace with respect to the matrix $\mH_k$, i.e., $\mH_k \vs \in \mathcal{V}_k$ for any $\vs\in \mathcal{V}_k$. 
  \end{enumerate}
\end{proposition}
From Proposition \ref{prop:maximal_krylov}, we immediately obtain that $\mH_k \vg_k \in \mathcal{V}_k$. Moreover, by repeatedly applying~(ii), it follows from induction that $\mH_k^i \vg_k \in \mathcal{V}_k$ for any $i \geq 0$. Thus, we conclude that the minimal subspace satisfying both conditions (A) and (B) is given by the linear span of $\{\mH_k^i \vg_k\}_{i=0}^{\infty}$, %
which is exactly the maximal \emph{Krylov subspace} generated by $\mH_k$ and $\vg_k$. Formally, 
the $j$-th Krylov subspace generated by a symmetric matrix $\mA$ and a vector $\vb$ is defined as
\begin{equation*}
    \mathcal{K}_j (\mA, \vb) = \mathrm{span}\{\vb, \mA \vb, \dots, \mA^{j-1} \vb\}. 
\end{equation*} 
Moreover, it can be shown that \citep[Chapter 6]{saad2011numerical} there exists an integer $r_0$ such that $\mathcal{K}_{j}(\mA, \vb) = \mathcal{K}_{r_0}(\mA, \vb)$ for all $j \geq r_0$, and we call $\mathcal{K}_{r_0}(\mA, \vb)$ as the maximal Krylov subspace. {In this case, the dimension of the maximal Krylov subspace is $r_0$.}

Now given these definitions, if we select $\mathcal{V}_k$ in \eqref{eq:general_subspace} as the maximal Krylov subspace generated by $\mH_k$ and $\vg_k$ denoted by $\calK_{r_0}(\mH_k,\vg_k)$, then the key inequality in \eqref{eq:key_inequality} remains valid and we retain the same convergence rate as the full-dimensional CRN method. However, the dimension of the maximal Krylov subspace can be as large as $d$, which contradicts our goal of dimension reduction. To solve this problem, we propose using the Krylov subspace up to dimension $m$, which is $\calK_m(\mH_k,\vg_k)$. In this case, Condition (A) still holds exactly since $\vg_k \in \calK_m(\mH_k,\vg_k)$ for any $m\geq 1$, and the only violated condition due to this approximation is Condition (B). In fact, as we shall show in Section~\ref{sec:convergence}, the approximation error corresponding to Condition (B) would be independent of the problem's dimension $d$. In comparison, when using random subspace selection as in SSCN, both conditions in (A) and (B) only hold approximately and the induced approximation error depends on the ratio $m/d$, resulting in a convergence rate that inevitably depends on $d$.

\subsection{Krylov Cubic Regularized Newton}
Next, we present our proposed Krylov CRN method, which is a particular instance of the subspace CRN method in \eqref{eq:general_subspace}. Specifically, we choose the subspace~$\calV_k$ to be the $m$-th Krylov subspace $\calK_m(\mH_k,\vg_k)$ and let $\mV_k \in \mathbb{R}^{d\times m}$ be the matrix formed by its orthonormal basis. 
Our method is summarized in Algorithm~\ref{alg:subspace_cubic}. 

\begin{algorithm}[t!]
  \caption{Krylov cubic regularized Newton}\label{alg:subspace_cubic}
  \begin{algorithmic}[1]
  \State \textbf{Input:}  Initial point $\vx_0 \in \mathbb{R}^d$, subspace dimension $m$, regularization parameter $M>0$
  \For{$k=0,1,\dots,$}
  \State Set $(\mV_k, \tilde{\vg}_k, \tilde{\mH}_k) = \textsc{Lanczos}(\mH_k, \vg_k;m)$ %
  \State 
  Solve the cubic subproblem 
  \begin{equation*}
      \vz_{k} = \argmin_{\vz \in \mathbb{R}^m}\; \Bigl\{ \tilde{\vg}_k^\top \vz  + \frac{1}{2} \vz^\top \tilde{\mH}_k \vz  + \frac{M}{6}\|\vz\|^3 \Bigr\},
  \end{equation*}
  \State Update $\vx_{k+1} = \vx_k + \mV_k \vz_k$
  \EndFor
  \end{algorithmic}
  \end{algorithm}

The only missing piece in Algorithm~\ref{alg:subspace_cubic} is how to compute the orthonormal basis of $\calV_k$ to form the matrix $\mV_k$, as well as the subspace gradient $\tilde{\vg}_k$ and the subspace Hessian $\tilde{\mH}_k$. 
To achieve this goal, we use the Lanczos method as shown in Algorithm~\ref{alg:Lanczos}, which is a commonly used method for computing an orthonormal basis for the Krylov subspace \citep{lanczos1950iteration}. 
  \begin{algorithm}[t!]
    \caption{$(\mV,\tilde{\vb},\tilde{\mA} )=\textsc{Lanczos}(\mA, \vb;m)$}\label{alg:Lanczos}
\begin{algorithmic}[1]
\State \textbf{Input:} $\mA \in \mathbb{R}^{d \times d}$, $\vb \in \mathbb{R}^d$, and the dimension $m$
\State \textbf{Initialize:} $\vv_1 = {\vb}/{\|\vb\|}$, $\beta_1 = 0$, $\vv_0 = 0$ 
\For{$j=1,2,\dots,m$}
    \State $\vw_j \leftarrow \mA \vv_j - \beta_j \vv_{j-1}$
    \State $\alpha_j \leftarrow  \vw_j^\top \vv_j  $
    \State $\vw_j \leftarrow \vw_j - \alpha_j \vv_j $ 
    \State $\beta_{j+1} \leftarrow \|\vw_j\|_2$ 
    \State $\vv_{j+1} \leftarrow \vw_j / \beta_{j+1}$
\EndFor
\State \textbf{Output:} $\mV = [\vv_1, \vv_2, \dots, \vv_m]$, $\tilde{\vb} = \|\vb\| \ve_1$, and $\tilde{\mA} = \mathrm{tridiag}(\{\beta_j\}_{j=2}^{m}, \{\alpha_j\}_{j=1}^m, \{\beta_j\}_{j=2}^{m})$
\end{algorithmic}
\end{algorithm} 
Specifically, given an input matrix $\mA$, an input vector $\vb$ and the target dimension $m$, 
Algorithm~\ref{alg:Lanczos} iteratively generates a sequence of orthonormal vectors $\{\vv_j\}_{j=1}^m$ that spans the Krylov subspace, as well as 
two auxiliary scalar sequences $\{\alpha_j\}_{j=1}^m$ and $\{\beta_j\}_{j=2}^{m+1}$.

In the analysis, we will heavily rely on some useful properties of the Lanczos vectors $\{\vv_j\}_{j=1}^m$, which we summarize in Proposition~\ref{prop:Lanczos} for convenience. 
\begin{proposition}\label{prop:Lanczos}
    The following statements hold true. 
    \begin{enumerate}[(i),leftmargin=1cm]
    \vspace{-2mm}
      \item $\mathcal{K}_j(\mA, \vb) = \mathrm{span}\{\vv_1, \dots, \vv_j\}$, for any $j\geq 1$.
      \item $\mA \vv_j = \beta_{j+1} \vv_{j+1} +\alpha_j \vv_j + \beta_j \vv_{j-1}$, for any $j \geq 1$. 
      \item Let $\mV^{(j)} = [\vv_1, \vv_2, \dots, \vv_j]\in \mathbb{R}^{d\times j}$ and $\tilde{\mA}^{(j)} = \mathrm{tridiag}(\{\beta_l\}_{l=2}^{j}, \{\alpha_l\}_{l=1}^{j}, \{\beta_l\}_{l=2}^{j})$ defined as
        \begin{equation*} 
            \begin{bmatrix}
                \alpha_1    & \beta_2   &  \\
                \beta_2     & \alpha_2  & \beta_3 \\
                            & \beta_3   & \alpha_3    & \ddots \\
                            &           & \ddots    & \ddots    & \beta_{j-1} \\
                            &           &           & \beta_{j-1}   & \alpha_{j-1}  & \beta_j \\
                            &   &   &   & \beta_j   & \alpha_j
            \end{bmatrix}.
        \end{equation*}
        Then, we have that $(\mV^{(j)})^\top \mA \mV^{(j)} = \tilde{\mA}^{(j)}$ and $(\mV^{(j)})^\top \vb = \|\vb\|\ve^{(j)}_1$, where $\ve^{(j)}_1$ denotes the first standard basis vector in $\mathbb{R}^j$.  
    \end{enumerate}
\end{proposition}

Given the results of Proposition~\ref{prop:Lanczos}, by applying Algorithm~\ref{alg:Lanczos} with $\mH_k$ and $\vg_k$ as the inputs, we can obtain the orthogonal basis matrix $\mV_k$ for the Krylov subspace $\calK_m(\mH_k,\vg_k)$. As a byproduct, we also obtain the subspace gradient $\tilde{\vg}_k = \mV_k^\top \vg_k$ and the subspace Hessian $\tilde{\mH}_k = \mV_k^\top \mH_k \mV_k$ by (iii) in Proposition~\ref{prop:Lanczos}. Moreover, utilizing the Krylov subspace has the additional benefit of having $\tilde{\mH}_k$ as a tridiagonal matrix, which simplifies solving the cubic subproblem in \eqref{eq:general_subspace} greatly. To elaborate, following the standard approach as described in Section~\ref{subsec:subspace_newton}, each Newton step requires solving a tridiagonal system of linear equations, which can be solved in $\calO(m)$ operations. Thus, given $\tilde{\vg}_k$ and $\tilde{\mH}_k $, the cost of solving the cubic subproblem for our method is $\calO(m)$.

\begin{remark}[Computational cost of Algorithm~\ref{alg:subspace_cubic}]
 It is worth noting that in Algorithm~\ref{alg:Lanczos}, we only need to evaluate matrix-vector products of the matrix $\mA$. Therefore, in the implementation of Algorithm~\ref{alg:subspace_cubic}, we do not have to store the Hessian matrix $\mH_k$ explicitly, but only need to compute $m$ Hessian-vector products. This computation can be done efficiently via back-propagation with a computational cost similar to gradient computation \citep{pearlmutter1994fast}, which typically requires $\calO(d)$ arithmetic operations.  Hence, the total cost per iteration of our method is $\calO(md)$.
 \end{remark}

\section{Convergence Analysis}\label{sec:convergence}
In this section, we analyze the convergence rate of our proposed method in Algorithm~\ref{alg:subspace_cubic}. We first characterize the additional approximation error when solving the cubic subproblem over a Krylov subspace, as opposed to using the full-dimensional space $\reals^d$ as in \eqref{eq:cubic_newton}. This error analysis is key to our convergence proofs and is fully discussed in Section~\ref{subsec:error_krylov}. We then present our main theorems in Section~\ref{subsec:main}, covering both the cases where $f$ is convex and strongly convex. Finally, in Section~\ref{subsec:hessian_structure}, we highlight the connection between our convergence bounds and the eigenspectrum of the Hessian matrices. In particular, we show that our convergence rate can be further improved when the Hessian matrices admit specific spectral structures.

\subsection{Approximation Error of Krylov Subspace}\label{subsec:error_krylov}
As we discussed in Section~\ref{subsec:rationale}, the crucial step in the convergence analysis is to establish a similar upper bound on $f(\vx_{k+1})$ as in \eqref{eq:key_inequality}. We already showed that this key inequality will remain valid if Conditions (A), (B), and (C) are all satisfied, which is the case when the subspace $\calV_k$ in Algorithm~\ref{alg:subspace_cubic} is chosen as the maximal Krylov subspace.
However, since we employ a Krylov subspace only up to dimension $m$ in Algorithm~\ref{alg:subspace_cubic}, Condition (B) is potentially violated and this is the main source of 
the approximation error.

As it turns out, in our convergence analysis, it is enough to control the ``minimal violation'' of Condition~(B) among all the Krylov subspaces up to the dimension $m$. To formalize this, let us define $\mV_k^{(j)} \in \reals^{d \times j}$ as the matrix that consists of the first $j$ Lanczos vectors (cf. item (iii) of Proposition~\ref{prop:maximal_krylov}). Then, we can define $\mP_k^{(j)} = \mV_k^{(j)} \mV_k^{(j)\top}$, which is the orthogonal projection matrix of the subspace $\mathcal{K}_j(\mH_k,\vg_k)$. %
In the following lemma, we provide the key inequality for Algorithm~\ref{alg:subspace_cubic} that serves a similar role as the one in \eqref{eq:key_inequality}.   
\begin{lemma}\label{lem:intermediate_step}
    Let $\{\vx_k\}_{k\geq 0}$ be the sequence generated by Algorithm~\ref{alg:subspace_cubic} and suppose that Assumptions~\ref{assum:convex}, \ref{assum:bounded_sublevel} and \ref{assum:hessian} hold true. Suppose $M\geq L_2$. For any $\vs \in \mathbb{R}^d$, we have 
    \begin{equation*}
        f(\vx_{k+1}) \leq f(\vx_k) +  \vg_k^\top \vs  + \frac{1}{2}  \vs^\top \mH_k  \vs + \frac{M}{6}\| \vs\|^3+\frac{1}{2}\min_{j\in\{ 1,\dots,m\}} \left\{ (\mP_k^{(j)}\vs)^\top \mH_k \mP_k^{(j)}\vs -\vs^\top \mH_k \vs \right\}. \nonumber%
    \end{equation*}
\end{lemma}
In light of Lemma~\ref{lem:intermediate_step}, 
our aim is to bound the additional error term $\min_{j\in \{1,\dots,m\}} \{ (\mP_k^{(j)}\vs)^\top \mH_k \mP_k^{(j)}\vs -\vs^\top \mH_k \vs \}$ for any $\vs\in \mathbb{R}^d$.
Before stating our result in Lemma~\ref{lem:bessel_main}, we introduce a quantity that 
plays a major role in our error analysis. 
Specifically, for a symmetric matrix $\mA\in \mathbb{R}^{d\times d}$ and a vector $\vb \in \mathbb{R}^d$, we define 
\begin{equation*}
  \sigma^{(j)}(\mA,\vb) := \max_{\vw \in \calK^{(j)}, \|\vw\|=1} {\mathrm{dist}(\mA \vw, \calK^{(j)})}
\end{equation*}
for $j=1,\dots,m$, where $\calK^{(j)}$ denotes $ \calK_j(\mA,\vb)$
and $\mathrm{dist}(\vu, \calV) := \min_{\vv \in \calV} \|\vu-\vv\|$ is the distance between a vector $\vu$ and a subspace $\calV$. Intuitively, $\sigma^{(j)}(\mA,\vb)$ characterizes how far a vector $\vw$ in $\calK^{(j)}$ can be pushed away from the subspace under the linear mapping $\mA$, and in particular we have $\sigma^{(j)}(\mA,\vb)=0$ if and only if $\calK_j(\mA,\vb)$ is an invariant subspace with respect to the matrix $\mA$. Furthermore, we define 
\begin{equation}\label{eq:def_rho}
        \rho^{(m)}(\mA,\vb) := 
        \biggl(\prod_{j=1}^m \sigma^{(j)}(\mA,\vb)\biggr)^{\frac{1}{m}}.
\end{equation}
As we shall discuss in Section~\ref{subsec:hessian_structure}, $\rho^{(m)}(\mA,\vb)$ is closely related to the eigenspectrum of the matrix $\mA$. In particular, it can be upper bounded by the largest eigenvalue of $\mA$ in the worst case, as shown in Lemma~\ref{lem:L1_bound}. 
\begin{lemma}\label{lem:L1_bound}
  Assume that $0 \preceq \mA \preceq L_1 \mI$. Then, for any $\vb\in \reals^d$ we have $\rho^{(m)}(\mA,\vb) \leq 2^{1/m} L_1/4$. 
\end{lemma}

As a corollary, 
we emphasize that 
$\rho^{(m)}(\mA,\vb)$ is independent of the problem's dimension and can be upper-bounded by the largest eigenvalue of $\mA$. Moreover, this upper bound can be further improved if the matrix $\mA$ exhibits specific spectral structures. We defer the discussions regarding these special cases to Section~\ref{subsec:hessian_structure}.

Next we leverage the definition in \eqref{eq:def_rho} to establish an upper bound on the approximation error caused by using a Krylov subspace of dimension $m$. 
\begin{lemma}\label{lem:bessel_main}
For any $k\geq 0$ and any $\vs\in \mathbb{R}^d$, we have 
  \begin{equation*}
      \min_{j \in \{1,\dots, m\}} \left\{ (\mP_k^{(j)}\vs)^\top \mH_k \mP_k^{(j)}\vs -\vs^\top \mH_k \vs \right\}
      \leq  \frac{2\rho^{(m)}(\mH_k,\vg_k)}{m}\|\vs\|^2.
  \end{equation*}
\end{lemma}

By combining Lemma~\ref{lem:intermediate_step}
and Lemma~\ref{lem:bessel_main}, we obtain 
\begin{equation}
    f(\vx_{k+1}) \leq f(\vx_k) +  \vg_k^\top \vs  + \frac{1}{2}  \vs^\top \mH_k  \vs+ \frac{M}{6}\| \vs\|^3 + \frac{\rho^{(m)}(\mH_k,\vg_k)}{m}\|\vs\|^2 . \label{eq:key_inequality_krylov}
\end{equation}
Compared with \eqref{eq:key_inequality}, the inequality in \eqref{eq:key_inequality_krylov} has an additional error term $\rho^{(m)}(\mH_k,\vg_k)\|\vs\|^2/m$. 
Given that $\rho^{(m)}(\mH_k,\vg_k)$ can be upper bounded by a constant as shown in Lemma~\ref{lem:L1_bound}, we observe that the error term decays at the rate of $\bigO(1/m)$ as the subspace dimension $m$ increases. Thus, as we should expect, a larger subspace dimension results in a smaller approximation error, which then leads to a faster convergence rate.  %

\subsection{Main Results}\label{subsec:main}
In this section, we will leverage the upper bounds established in the previous section to present the convergence rate of our proposed method in both convex and strongly convex settings. To this end, {let $\vx^*$ be an optimal solution of $f$ and}  define $D$ as
\begin{equation}\label{eq:def_D}
  D := \sup \{\|\vx-\vx^*\|:\; \vx\in \reals^d, \;f(\vx) \leq f(\vx_0)\},
\end{equation}
which is finite since 
the level-set $\{\vx\in \mathbb{R}^d: f(\vx) \leq f(\vx_0)\}$ is bounded due to Assumption \ref{assum:bounded_sublevel}.  In addition, the following quantity will be essential in measuring the convergence rate of our proposed algorithm
    \begin{equation}\label{eq:def_rho_max}
        \rho^{(m)}_{\max{}} = \max_{i\in\{{0,1,\dots,k-1}\}}\{\rho^{(m)}(\mH_{i}, \vg_i)\},
    \end{equation}
    as it provides a universal upper bound on the error corresponding to Condition (B). 
    {By Lemma~\ref{lem:L1_bound}, if we assume that $\nabla^2 f(\vx_k) \preceq L_1 \mI$ as in \citep{Hanzely2020}, then we have $\rho^{(m)}_{\max{}} \leq 2^{1/m} L_1/4$. On the other hand, $\rho^{(m)}_{\max{}}$ can be much smaller than $L_1$ in some special cases as we discuss in Section~\ref{subsec:hessian_structure}.} 
    Next we present the convergence rate of our method in the convex setting. 
\begin{theorem}\label{thm:main}
    Let $\{\vx_k\}_{k\geq 0}$ be the sequence generated by Algorithm~\ref{alg:subspace_cubic} and suppose that Assumptions~\ref{assum:convex}, \ref{assum:bounded_sublevel} and \ref{assum:hessian} hold true. Then, for any $k \geq 0$, we have
    \begin{equation*}
        f(\vx_k) - f(\vx^*) \leq \frac{9 \rho^{(m)}_{\max{}} D^2}{2m} \left(\frac{1}{k}+\frac{1}{k^2}\right) + \frac{9(L_2+M)D^3}{2k^2}.\!
    \end{equation*}
\end{theorem}
Theorem~\ref{thm:main} shows that $f(\vx_k)-f(\vx^*)$ converges at a global convergence rate of $\calO(\frac{1}{mk}+\frac{1}{k^2})$. 
As a corollary, to reach an error of $\epsilon>0$, the number of required iterations can be upper bounded by $\calO(\frac{1}{m\epsilon}+\frac{1}{\sqrt{\epsilon}})$, which is indeed independent of the problem's dimension $d$. To the best of our knowledge, this is the first subspace second-order method that attains a dimension-independent convergence rate. {{In comparison, SSCN in \citep{Hanzely2020} achieves an iteration complexity of $\calO(\frac{d-m}{m}\cdot\frac{1}{\epsilon}+\frac{d}{m}\cdot\frac{1}{\sqrt{\epsilon}})$.
{Note that when we deal with subspace second-order methods, we are particularly interested in scenarios where the subspace dimension $m$ is a fixed constant much smaller than the problem dimension $d$, i.e., $m\ll d$. This is driven by a practical necessity to ensure a scalable per-iteration computational cost as the problem dimension increases. As our convergence bounds suggest, in the regime where $m \ll d$, our iteration complexity is lower than the one for SSCN by a factor of $d$.}
 }}

Moreover, a better complexity can be achieved if the function $f$ is strongly convex with parameter $\mu>0$, that is, $\nabla^2 f(\vx) \succeq \mu \mI$ for all $\vx\in \reals^d$. As we show in Theorem~\ref{thm:strongly-convex}, in this setting Algorithm~\ref{alg:subspace_cubic} can achieve a linear rate of convergence. 

\begin{theorem}\label{thm:strongly-convex}
    Let $\{\vx_k\}_{k\geq 0}$ be the sequence generated by Algorithm~\ref{alg:subspace_cubic} and suppose Assumptions~\ref{assum:convex}, \ref{assum:bounded_sublevel}, and \ref{assum:hessian} hold true. Moreover, assume that $f$ is $\mu$-strongly convex. Then, the number of iterations required to reach $\delta_{k} := f(\vx_k)-f(\vx^*) \leq \epsilon$ can be upper bounded by 
    \begin{equation*}
        k = \bigO\Biggl(\biggl(\frac{\rho_{\max}^{(m)}}{m \mu}+1\biggr)\log\frac{\delta_{0}}{\epsilon} + \frac{\sqrt{L_2+M}\delta_{0}^{0.25}}{\mu^{0.75}}\Biggr).
    \end{equation*}
\end{theorem}

\begin{remark}
    {Similar to \citet{Hanzely2020}, Theorem~\ref{thm:strongly-convex} continues to hold if we replace the strong convexity of $f$ with the weaker assumption that $\frac{\mu}{2}\|\vx-\vx^*\|^2 \leq f(\vx) - f(\vx^*)$ for any $\vx \in \mathbb{R}^d$.}
\end{remark}
When the target accuracy $\epsilon$ is sufficiently small, Theorem~\ref{thm:strongly-convex} shows that the iteration complexity is dominated by $\frac{\rho_{\max}^{(m)}}{m \mu} \log \frac{1}{\epsilon}$, where the factor $\frac{\rho_{\max}^{(m)}}{m \mu}$ can be regarded as the effective condition number. 
In particular, consider the special case where 
 $\mu \mI \preceq \nabla^2 f(\vx) \preceq L_1 \mI$ for all $\vx\in \reals^d$. By Lemma~\ref{lem:L1_bound}, we always have $\frac{\rho_{\max}^{(m)}}{m \mu} \leq \frac{L_1}{m \mu}$. On the other hand, under the same setting as in \citep{Hanzely2020}, SSCN achieves an iteration complexity of $\calO \left(\left(\frac{d-m}{m}\frac{L_1}{\mu}+\frac{d}{m}\right) \log \frac{1}{\epsilon} \right)$. Thus, we shave a factor of $d$ from their iteration complexity bound when $m= \calO(1)$. Moreover, as we discuss in the next section, the quantity $\rho^{(m)}_{\max}$ can be much smaller than $L_1$ when the Hessian admits some specific spectral structure, leading to an even faster convergence rate. 

\subsection{Special Cases}\label{subsec:hessian_structure}
As shown in Theorems~\ref{thm:main} and~\ref{thm:strongly-convex}, the quantity $\rho^{(m)}_{\max}$ plays an important role in the convergence rate of our proposed method. {Moreover, by its definition in \eqref{eq:def_rho_max}, we can upper bound $\rho^{(m)}(\mathbf{H}_i,\bm{g}_i)$ for $i=0,1,\dots,k-1$ individually and then take the maximum to establish an upper bound on $\rho^{(m)}_{\max}$. }
In this section, we will show that $\rho^{(m)}(\mH_i,\vg_i)$ is closely related to the structure of the eigenspectrum of the Hessian matrix $\mH_i$, {and it can be much smaller than the worst-case upper bound in Lemma~\ref{lem:L1_bound} in some special cases.} The proofs in this section are deferred to Appendix~\ref{appen:rho_m}.

To begin with, we derive an alternative expression for $\rho^{(m)}(\mA,\vb)$ in \eqref{eq:def_rho} in terms of matrix polynomials. {
Specifically, for a polynomial of degree $m$ given by $p(x) = x^m+ \sum_{i=0}^{m-1} c_i x^i$, we define $p(\mA) :=  \mA^m+ \sum_{i=0}^{m-1} c_i \mA^i$, where we observe the convention that $\mA^0 = \mI$.  Moreover, we define $\mathcal{M}_m:=\{x^m+\sum_{i=0}^{m-1} c_i x^i \mid c_0,\dots,c_{m-1} \in \reals\}$ as the set of monic polynomials of degree $m$. With these notations, we are ready to state Lemma~\ref{lem:matrix_polynomial}.}  
\begin{lemma}\label{lem:matrix_polynomial}
    The quantity $\rho^{(m)}(\mA,\vb)$ in \eqref{eq:def_rho} can be equivalently expressed as 
    \begin{equation*}
        \rho^{(m)}(\mA,\vb) = \min_{c_0,\dots,c_{m-1} \in \mathbb{R}} \left\|\left(\mA^m + \sum_{i=0}^{m-1} c_i \mA^{i}\right)\frac{\vb}{\|\vb\|} \right\|^{\frac{1}{m}} = \min_{p \in \mathcal{M}_m } \left\|p(\mA) \frac{\vb}{\|\vb\|} \right\|^{\frac{1}{m}}. 
    \end{equation*}

\end{lemma}

Lemma~\ref{lem:matrix_polynomial} implies that $\rho^{(m)}(\mA,\vb) \leq \min_{p \in \mathcal{M}_m } \left\|p(\mA) \right\|_{\op}^{\frac{1}{m}}$, where $\|\cdot\|_{\op}$ denotes the operator norm of the matrix. Furthermore, assume that $\mA$ has $r$ distinct eigenvalues $\lambda_1 > \lambda_2 > \dots > \lambda_r$, where $r\leq d$. Then, we have 
\begin{equation*}
    \rho^{(m)}(\mA,\vb) \leq \min_{{p \in \mathcal{M}_m} }\max_{i\in\{1,2,\dots,r\}} \; |p(\lambda_i) |^{\frac{1}{m}}.
\end{equation*}
Hence, by making suitable assumptions on the eigenvalues $\lambda_1,\lambda_2,\dots,\lambda_r$ and choosing the polynomial~$p$ accordingly, we can obtain different bounds on $\rho^{(m)}$.

\begin{remark}
    {It is worth noting that the matrix approximation problem $\min_{p \in \mathcal{M}_m} \|p(\mA)\|_{\op}$ has been studied before \citep{greenbaum1994gmres}, and the polynomial that achieves the minimum is also known as the \emph{Chebyshev polynomial of} $\mA$ \citep{toh1998chebyshev,faber2010chebyshev}. However, to the best of our knowledge, no explicit upper bound on its optimal value has been reported in the literature. }
\end{remark}

\noindent \textbf{Example I.} 
Our first result shows that $\rho^{(m)}$ can be upper bounded by the geometric mean of the largest $m$ distinct eigenvalues of the Hessian $\mH$.

\begin{lemma}\label{lem:eigen_decay}
   Assume that the Hessian $\mH$ has $r$ distinct eigenvalues in decreasing order: $\lambda_1 > \lambda_2 > \dots > \lambda_r$.
  For any $\vg\in \reals^d$, $\rho^{(m)}(\mH,\vg) \leq \left(\prod_{i=1}^m \lambda_i\right)^{\frac{1}{m}}$ when $m< r$ and $\rho^{(m)}(\mH,\vg) = 0$ when $m\geq r$.  
\end{lemma}
Hence, if the eigenvalues of $\mH$ decay fast, then the bound in Lemma~\ref{lem:eigen_decay} can be much smaller than the one in Lemma~\ref{lem:L1_bound}. 
Moreover, as an important corollary, 
if the number of nonzero eigenvalues of the Hessian is at most $m-1$, then $\rho^{(m)}(\mH,\vg)=0$ and the convergence rate of our algorithm becomes $\calO(1/k^2)$. This means that, as a special case, if the rank of Hessian is at most $m-1$, the convergence rate of our Krylov CRN method with dimension $m$ is as good as the full-dimensional CRN method. On the other hand, random subspace methods such as SSCN are unable to recover such a result when the Hessian has a low-rank structure. 

\begin{remark}\label{rem:maximal_Krylov}
    {The second result in Lemma~\ref{lem:eigen_decay} can be further strengthened: we have $\rho^{(m)}(\mH,\vg)=0$ \emph{if and only if} $m \geq r_0$, where $r_0$ denotes the dimension of the maximal Krylov subspace generated by $\mH$ and $\vg$.
    The proof can be found in Appendix~\ref{appen:eigen_decay}. %
    } 
\end{remark}

\noindent \textbf{Example II.}
The second case we study is when the eigenvalues of the Hessian $\mH$ are concentrated in two clusters $[0,\Delta]$ and $[L_1-\Delta,L_1]$, where $\Delta$ is a small constant (see \citet{goujaud2022super}).
\begin{lemma}\label{lem:cluster}
  Assume that all the eigenvalues of the Hessian $\mH$ lie within the two intervals $[0,\Delta]$ and $[L_1-\Delta, L_1]$ for some $L_{1} > \Delta > 0$. Then, when $m$ is even, for any $\vg\in \reals^d$ we have $\rho^{(m)}(\mH,\vg) \leq 2^{1/m}\sqrt{\Delta(L_1-\Delta)}/2$. 
\end{lemma}

Lemma~\ref{lem:cluster} shows that if the eigenvalues of the Hessian are close to either $0$ or $L_1$ with an error of size~$\Delta$, then the convergence rate of our method becomes $\calO\left(\frac{\Delta(L_1-\Delta)}{mk}+\frac{1}{k^2}\right)$. As a result, to reach $\epsilon$-accuracy, the overall iteration complexity of our method in the convex setting becomes $\calO\left( \frac{\Delta(L_1-\Delta) }{\epsilon }+\frac{1}{\sqrt{\epsilon}} \right) $. Thus, if the size of the cluster satisfies $\Delta=\calO(\sqrt{\epsilon})$, the overall iteration complexity of our method becomes $\calO\left( \frac{1}{\sqrt{\epsilon}} \right)$, matching the one for the classic CRN method.

\section{Numerical Experiments}\label{sec:numerical}

In this section, we compare the numerical performance of our Krylov CRN method with full-dimensional CRN \citep{nesterov2006cubic} and SSCN \citep{Hanzely2020}. We focus on logistic regression problems on LIBSVM datasets~\citep{chang2011libsvm}. In our experiments, we use three different datasets representing low, medium, and high-dimensional problems.  Specifically, we use ``w8a", ``rcv1", and ``news20", with dimensions $d=300$, $d=47,236$, and $d=1,355,191$, respectively. Note that the sample sizes $n$ of these datasets are almost comparable with each other. 
 We set the subspace dimension to be $m = 10$ in our method, while we vary the subspace dimension in SSCN from $10$ to $1,000$ to ensure its best performance. Moreover, since the Lipschitz constant of the Hessian is unknown in advance,  
we use a backtracking line search scheme to select the regularization parameters. All experiments are run on a MacBook Pro with an Apple M1 chip and 16GB RAM.  

\begin{figure}[t!]
    \vspace{-2mm}
    \centering
    \subfloat[]
    {\includegraphics[width=0.32\linewidth]{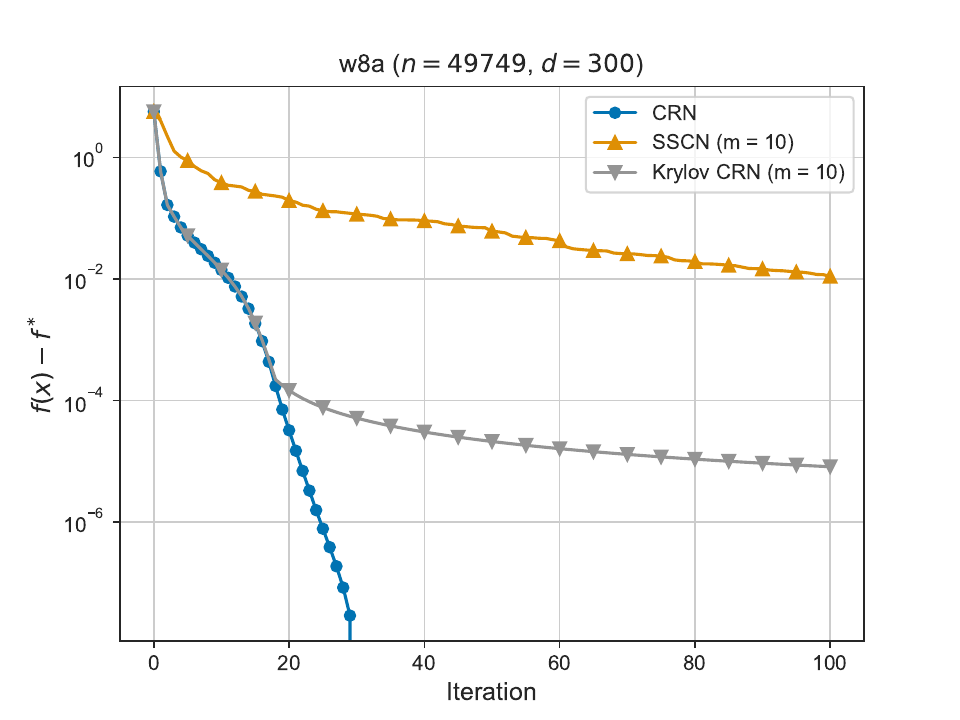}\label{fig:iteration_w8a}}
  \subfloat[]
  {\includegraphics[width=0.32\linewidth]{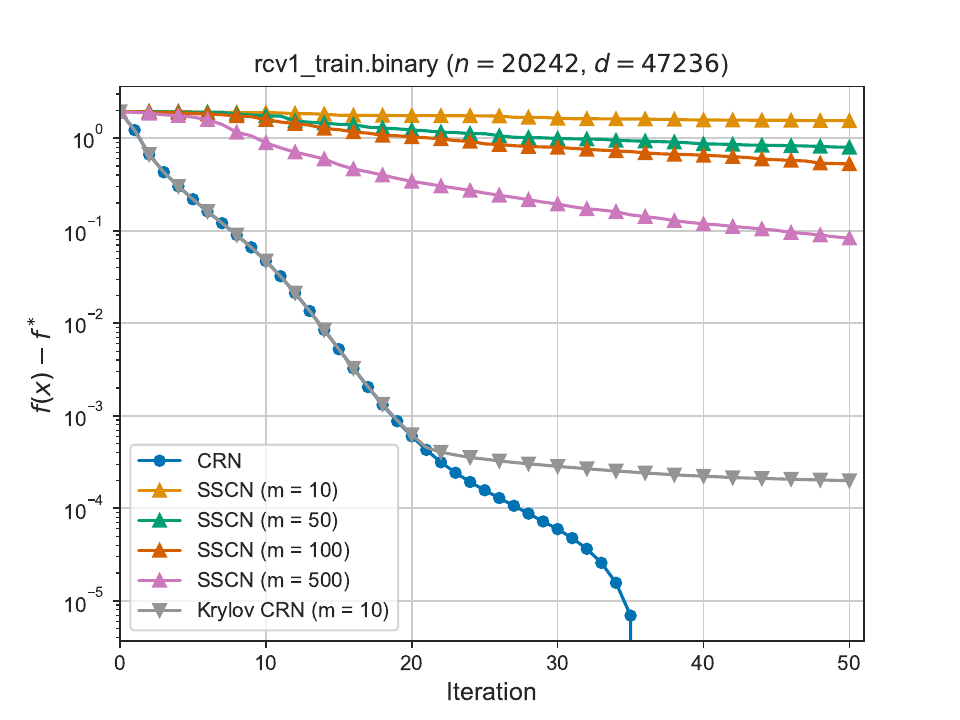}\label{fig:iteration_rcv1_train}}
  \subfloat[]
  {\includegraphics[width=0.32\linewidth]{iteration_news20.binary.pdf}\label{fig:iteration_news20}} \\[-.3em]
    \subfloat[]
    {\includegraphics[width=0.32\linewidth]{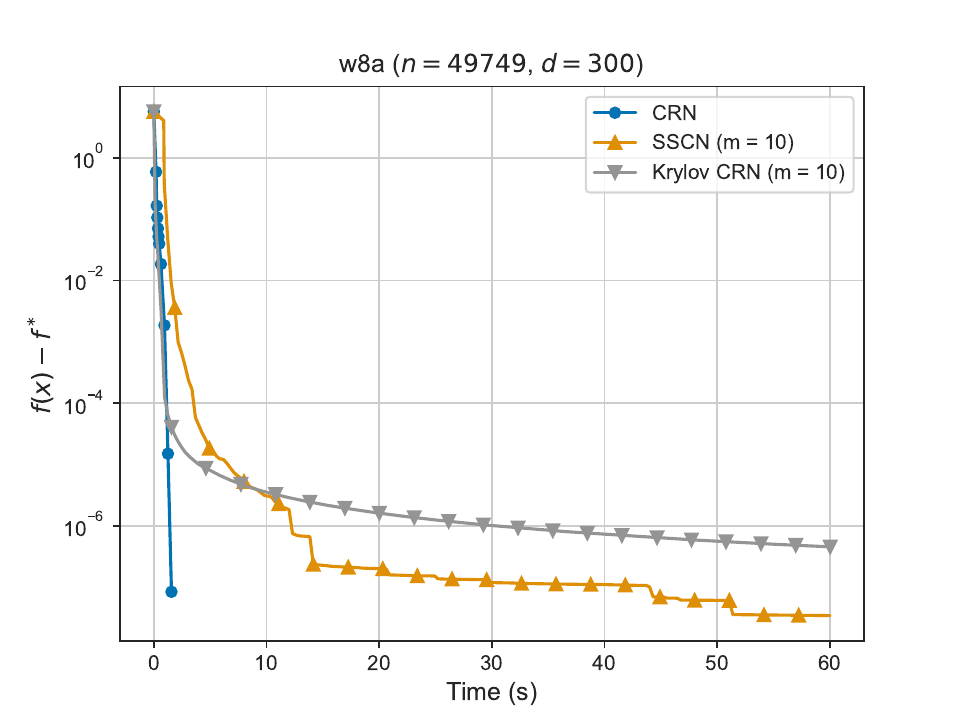}\label{fig:time_w8a}}
    \subfloat[]
    {\includegraphics[width=0.32\linewidth]{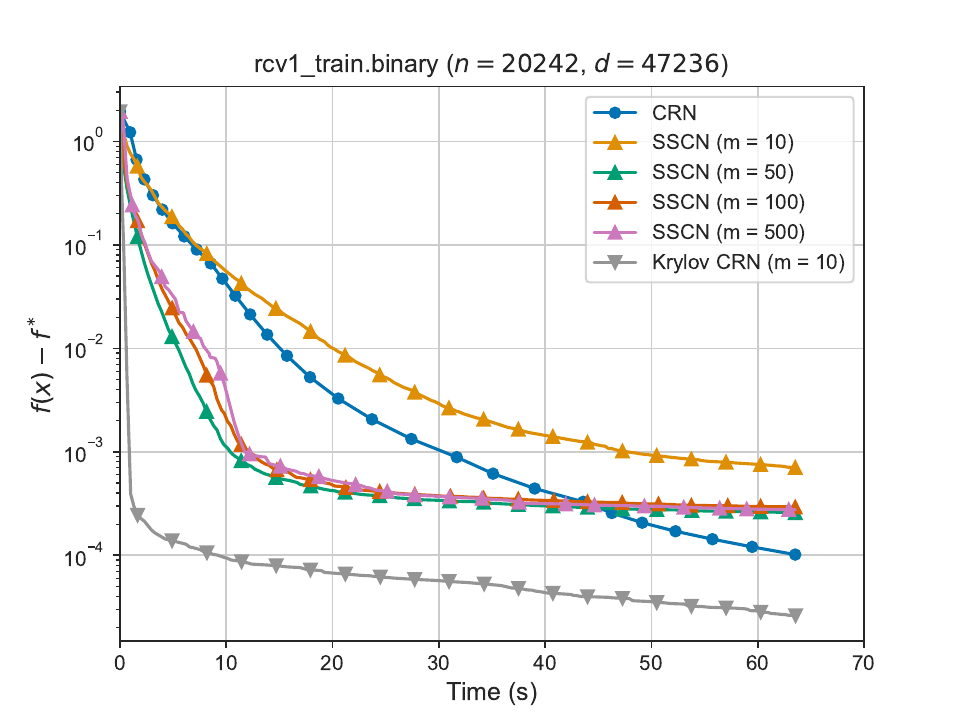}\label{fig:time_rcv1}}
    \subfloat[]
    {\includegraphics[width=0.32\linewidth]{intro_news20.binary.pdf}\label{fig:time_news20}}
    \caption{Numerical comparisons of CRN, SSCN, and our proposed Krylov CRN method for logistic regression on LIBSVM datasets. Here, $m$ is the subspace dimension, $n$ is the number of samples, and $d$ is the number of parameters. {The first row depicts the loss curves in terms of the number of iterations, while the second row depicts these curves in terms of the overall computation time.} We observe that Krylov CRN converges much faster than the others when the dimension $d$ is large.}
    \label{fig:logistic}
\end{figure}

In Figures~\subref{fig:iteration_w8a}-\subref{fig:iteration_news20}, we plot the suboptimality gap of CRN, SSCN, and our proposed Krylov CRN method in terms of the number of iterations. We observe that our method performs similarly to the full-dimensional CRN method, %
even though the cubic subproblem is solved over a subspace of dimension $m=10$. Also, compared with SSCN, our Krylov CRN method is able to make much more progress in each iteration, and the gap between these two methods becomes more significant as the problem's dimension $d$ increases. We should add that the convergence behaviors of these methods in Figures~\subref{fig:iteration_w8a}-\subref{fig:iteration_news20} in terms of iteration complexity are consistent with the theoretical results discussed in this paper. Specifically, we observe that SSCN converges slower as $d$ increases, which is expected given the fact that its convergence rate scales linearly with $d$. Conversely, the convergence path of our proposed method remains almost unchanged as $d$ increases. This consistency aligns with our theoretical findings that the iteration complexity of our method is dimension-free. 

Since the per-iteration costs of these methods are different, in Figures~\subref{fig:time_w8a}-\subref{fig:time_news20} we also compare their performance in terms of the overall computation time.
As shown in Figure~\subref{fig:time_w8a}, the full-dimensional CRN method converges very fast and outperforms all the other methods on a problem with a low-dimensional dataset, such as ``w8a''. However, as the dimension of the problem increases, the per-iteration cost of the CRN method rises rapidly and completes much fewer iterations in the given amount of time, resulting in its poor performance. 
{Moreover, while SSCN typically incurs a smaller per-iteration cost than our method,}  Figures~\subref{fig:time_rcv1} and~\subref{fig:time_news20} demonstrate that our method still overall outperforms SSCN due to its faster convergence rate. 

\section{Conclusion}
In this paper, we proposed the Krylov subspace cubic regularized Newton method, where we perform the cubic regularized Newton update over the Krylov subspace associated with the Hessian and the gradient at the current iterate. We proved that our proposed method converges at a rate of $\bigO\left(\frac{1}{mk}+\frac{1}{k^2}\right)$ in the convex setting, where $m$ is the subspace dimension and $k$ is the iteration counter. Furthermore, we characterized how the spectral structure of the Hessian matrices impacts its convergence, showing that the convergence rate can improve when the Hessian admits favorable structures. Finally, our experiments demonstrated the superior performance of our proposed method.

\section*{Acknowledgements}
{This research of R. Jiang and A. Mokhtari is supported in part by NSF Grants 2007668, 2019844, and  2112471,  ARO  Grant  W911NF2110226,  the  Machine  Learning  Lab  (MLL)  at  UT  Austin, and the Wireless Networking and Communications Group (WNCG) Industrial Affiliates Program.}
\printbibliography

\newpage
\appendix
\input{appendix.tex}

\end{document}

%% file: math_commands.tex
\newcommand\Vector[1]{\mathbf{#1}}

\newcommand\va{{\Vector{a}}}
\newcommand\vb{{\Vector{b}}}

\newcommand\ve{{\Vector{e}}}

\newcommand\vg{{\Vector{g}}}

\newcommand\vs{{\Vector{s}}}

\newcommand\vu{{\Vector{u}}}
\newcommand\vv{{\Vector{v}}}
\newcommand\vw{{\Vector{w}}}
\newcommand\vx{{\Vector{x}}}
\newcommand\vy{{\Vector{y}}}
\newcommand\vz{{\Vector{z}}}

\newcommand\MATRIX[1]{\mathbf{#1}}

\newcommand\mA{{\MATRIX{A}}}

\newcommand\mH{{\MATRIX{H}}}
\newcommand\mI{{\MATRIX{I}}}

\newcommand\mP{{\MATRIX{P}}}

\newcommand\mV{{\MATRIX{V}}}

\newcommand\bigO{\mathcal{O}}

\DeclareMathOperator*{\E}{\mathbb{E}}
\newcommand\op{{\mathrm{op}}}

\DeclareMathOperator*{\argmin}{arg\,min}

%% file: appendix.tex
\section*{Appendix}
\section{Proof of Proposition~\ref{prop:maximal_krylov}}\label{appen:maximal_krylov}

To prove item (i), we note that Condition (A) is equivalent to $(\mP_k \vg_k -  \vg_k)^\top \vs \leq 0$ for any $\vs \in \mathbb{R}^d$. In particular, by choosing $\vs = \mP_k \vg_k - \vg_k$, we can deduce that $\vg_k = \mP_k \vg_k$. Since $\mP_k$ is an orthogonal projection matrix associated with the subspace $\mathcal{V}_k$, this holds if and only if $\vg_k \in \mathcal{V}_k$. 

To prove item (ii), we consider the following decomposition $\vs = \hat{\vs} + \vs^\perp$, where $\hat{\vs} \in \mathcal{V}_k$ and $\vs^\perp \in \mathcal{V}_k^{\perp}$. Here, $\mathcal{V}_k^{\perp}$ denotes the orthogonal complement of $\mathcal{V}_k$. Since $\hat{\vs} = \mP_k \vs$, we have 
\begin{align*}
  \vs^\top \mH_k \vs - (\mP_k \vs)^\top \mH_k (\mP_k \vs) 
 = \vs^\top \mH_k \vs - \hat{\vs}^\top \mH_k \hat{\vs} &= (\hat{\vs} + \vs^{\perp})^\top \mH_k (\hat{\vs} + \vs^{\perp}) - \hat{\vs}^\top \mH_k \hat{\vs} \\
  &= 2 (\vs^{\perp})^\top \mH_k \hat{\vs} + (\vs^{\perp})^\top  \mH_k \vs^{\perp}.
\end{align*}
Thus, Condition (B) is equivalent to 
\begin{equation}\label{eq:semidefinite_test}
    2 (\vs^{\perp})^\top \mH_k \hat{\vs} + (\vs^{\perp})^\top  \mH_k \vs^{\perp} \geq 0, \quad \forall  \hat{\vs} \in \mathcal{V}_k \; \text{and} \; \vs^\perp \in \mathcal{V}_k^{\perp}.
\end{equation}
Moreover, we claim that \eqref{eq:semidefinite_test} further implies
\begin{equation}\label{eq:zero_inner_product}
    (\vs^{\perp})^\top \mH_k \hat{\vs} = 0, \quad \forall \hat{\vs} \in \mathcal{V}_k \;\text{and} \;\vs^\perp \in \mathcal{V}_k^{\perp}.
\end{equation}
Otherwise, suppose there exist $\hat{\vs}_0 \in \mathcal{V}_k$ and $\vs^\perp_0 \in \mathcal{V}_k^{\perp}$ with $(\vs_0^{\perp})^\top \mH_k \hat{\vs}_0 \neq 0$. Without loss of generality, assume $(\vs_0^{\perp})^\top \mH_k \hat{\vs}_0 > 0$. 
Since $\calV_k$ is a linear subspace, we have $\lambda \hat{\vs}_0 \in \calV_k$ for any $\lambda \in \mathbb{R}$. However, when $\lambda \rightarrow -\infty$, we have 
\begin{equation*}
     2 (\vs_0^{\perp})^\top \mH_k (\lambda\hat{\vs}_0) + (\vs_0^{\perp})^\top  \mH_k \vs_0^{\perp} = 2\lambda (\vs_0^{\perp})^\top \mH_k \hat{\vs}_0 + (\vs_0^{\perp})^\top  \mH_k \vs_0^{\perp} \rightarrow -\infty,
\end{equation*}
which leads to a contradiction with \eqref{eq:semidefinite_test}. This proves that the statement in \eqref{eq:zero_inner_product}. 
Finally, since ${\vs}^\perp$ can be any vector in $\calV_k^{\perp}$, \eqref{eq:zero_inner_product} holds if and only if $\mH_k \hat{\vs} \in \mathcal{V}_k$. Since $\hat{\vs}$ is an arbitrary vector in $\calV_k$, by definition, this shows that $\calV_k$ is an invariant subspace.

\section{The Lanczos Method}
First, we provide the proof of Proposition~\ref{prop:Lanczos} in Section~\ref{subsec:lanczos}. Then in Section~\ref{subsec:technical}, we present additional technical lemmas that we will use in the convergence proof. 

\subsection{Proof of Proposition~\ref{prop:Lanczos}}\label{subsec:lanczos}
Item (i) is a standard property of the Lanczos method; for instance, see \citet[Proposition 6.5]{saad2011numerical}.  Item (ii) follows from Lines~4, 6, and 8 in the update rule of Algorithm~\ref{alg:Lanczos}. To prove item (iii), we first rewrite the recursive formula in item (ii) in a matrix form, as shown in the following lemma. 

\begin{lemma}\label{lem:recursive_matrix_form}
We have $\mA \mV^{(j)} = \mV^{(j)} \tilde{\mA}^{(j)} +\beta_{j+1} \vv_{j+1} (\ve_{j}^{(j)})^\top$, where $\ve_j^{(j)}$ is the $j$-th standard basis vector in $\mathbb{R}^j$.  
\end{lemma}
\begin{proof}
  It suffices to verify that the two matrices in the identity have the exactly same columns. For any $l \in \{1,\dots,j\}$, the $l$-th column of $\mA \mV^{(j)}$ is given by $\mA \vv_l$. On the other hand, by the definition of $\tilde{\mA}^{(j)}$, we can observe that the $l$-th column of the the right-hand side matrix  is given by $\beta_{l} \vv_{l-1} + \alpha_{l} \vv_l + \beta_{l+1} \vv_{l+1}$. Hence, the equality holds due to Item (ii) in Proposition~\ref{prop:Lanczos}.    
\end{proof}
Now we are ready to prove item (iii) in Proposition~\ref{prop:Lanczos}. 
Since the vectors $\{\vv_j\}_{j=1}^m$ are orthonormal, it holds that $(\mV^{(j)})^\top \mV^{(j)} = \mI$ and $\left(\mV^{(j)}\right)^\top \vv_{j+1} = 0$. Thus, by Lemma~\ref{lem:recursive_matrix_form}, we have $ (\mV^{(j)})^\top \mA \mV^{(j)} = (\mV^{(j)})^\top \mV^{(j)} \tilde{\mA}^{(j)} +\beta_{j+1} (\mV^{(j)})^\top \vv_{j+1} (\ve_{j}^{(j)})^\top  = \tilde{\mA}^{(j)}$. Moreover, note that we set $\vv_1 = \frac{\vb}{\|\vb\|}$ in Algorithm~\ref{alg:Lanczos}. Since $(\mV^{(j)})^\top \vv_1 = \ve_1^{(j)}$, we further obtain that $(\mV^{(j)})^\top \vb = \|\vb\|\ve_1^{(j)}$. This completes the proof. 

\subsection{Technical Lemmas}\label{subsec:technical}
In this section, we provide some additional useful results regarding the outputs of the Lanczos method in Algorithm~\ref{alg:Lanczos}. For the convenience of the reader, we briefly recap our notations.  
Given the input matrix $\mA$ and the input vector $\vb$, we let $\{\vv_j\}_{j=1}^m$, $\{\alpha_j\}_{j=1}^m$, $\{\beta_j\}_{j=2}^{m+1}$ be the sequences generated during the execution of Algorithm~\ref{alg:Lanczos}. Also recall that $\mV^{(j)} = [\vv_1, \vv_2, \dots, \vv_j]\in \mathbb{R}^{d\times j}$ and $\mP^{(j)} = \mV^{(j)} (\mV^{(j)})^\top$. 
\begin{lemma}\label{lem:beta_product}
The following statement holds true. 
\begin{enumerate}[(i),leftmargin=1cm]
  \item $(\mI-\mP^{(j)}) \mA \mP^{(j)} = \beta_{j+1} \vv_{j+1}\vv_j^\top$ for any $j \geq 1$. 
  \item  For any $j \geq 1$, we have
  \begin{equation}\label{eq:beta_and_sigma}
    \beta_{j+1} = \sigma^{(j)}(\mA,\vb) := \max_{\vw \in \calK^{(j)}, \|\vw\|=1} {\mathrm{dist}(\mA \vw, \calK^{(j)})},
  \end{equation}
  where $\calK^{(j)}$ denotes $ \calK_j(\mA,\vb)$
and $\mathrm{dist}(\vu, \calV) := \min_{\vv \in \calV} \|\vu-\vv\|$ is the distance between a vector $\vu$ and a subspace $\calV$.
  \item  We have  
  \begin{equation}\label{eq:product_betas}
    \prod_{j=2}^{m+1} \beta_j = \min_{\vu \in \mathcal{K}_m(\mA,\vb)} \|\mA^{m} \vv_1 -\vu\|. 
\end{equation}
\end{enumerate}

\end{lemma}
{We note that some results in Lemma~\ref{lem:beta_product} are likely known in the literature (see, e.g., \citet{parlett1998symmetric}), but for completeness we provide their proofs below.}
\begin{proof}
 To prove item (i), by using Lemma~\ref{lem:recursive_matrix_form} and the fact that $\mP^{(j)} = \mV^{(j)} (\mV^{(j)})^\top$, we obtain 
  \begin{equation*}
    \mA \mP^{(j)} \!=\!\mA \mV^{(j)} (\mV^{(j)})^\top \!\!= \!\mV^{(j)} \tilde{\mA}^{(j)}(\mV^{(j)})^\top +\beta_{j+1} \vv_{j+1} (\ve_{j}^{(j)})^\top (\mV^{(j)})^\top \!=\! \mV^{(j)} \tilde{\mA}^{(j)}(\mV^{(j)})^\top +\beta_{j+1} \vv_{j+1} \vv_j^\top.\! 
  \end{equation*}
  Moreover, since the vectors $\{\vv_j\}_{j=1}^m$ are orthonormal, it holds that $(\mI-\mP^{(j)}) \mV^{(j)} = \mV^{(j)} - \mV^{(j)}(\mV^{(j)})^\top \mV^{(j)} = 0$ and $(\mI-\mP^{(j)})\vv_{j+1} = \vv_{j+1} - (\mV^{(j)})^\top \mV^{(j)} \vv_{j+1} = \vv_{j+1}$. Hence, we conclude that $(\mI-\mP^{(j)})\mA \mP^{(j)} = \beta_{j+1} \vv_{j+1} \vv_j^\top$. 

  To prove item (ii), we will use the result we just proved in item (i). Since $\mP^{(j)}$ is the orthogonal projection matrix associated to $\calK_j(\mA,\vb)$, we have $\mathrm{dist}(\mA \vw, \calK^{(j)}) = \|(\mI - \mP^{(j)}) \mA \vw\|$ by the distance-minimizing property of orthogonal projection. Moreover, note that for any $\vw \in \calK_j(\mA,\vb)$, we have $\mP^{(j)} \vw = \vw$. Thus, we can write 
  \begin{equation*}
    \mathrm{dist}(\mA \vw, \calK^{(j)}) = \|(\mI - \mP^{(j)}) \mA \vw\| = \|(\mI - \mP^{(j)}) \mA \mP^{(j)}\vw\| = \beta_{j+1}\|\vv_{j+1} (\vv_j^\top \vw)\| = \beta_{j+1} |\vv_j^\top \vw|.
  \end{equation*} 
  Thus, the maximum in \eqref{eq:beta_and_sigma} is achieved when $\vw = \pm \vv_j$, which proves the desired result. 

  To prove item (iii), we will first use \eqref{eq:beta_and_sigma} to derive the following alternative expression for $\beta_{j+1}$:   
  \begin{equation*}
    \beta_{2} = {\min_{\vu \in \mathcal{K}_1(\mA,\vb)} \frac{}{} \|\mA \vv_1 -\vu\|} \quad \text{and}\quad \beta_{j+1} = \frac{\min_{\vu \in \mathcal{K}_j(\mA,\vb)} \|\mA^{j} \vv_1 -\vu\|}{\min_{\vu \in \mathcal{K}_{j-1}(\mA,\vb)} \|\mA^{j-1} \vv_1 -\vu\|}\; \text{for }j\geq 2. 
  \end{equation*}
  If this is the case, then \eqref{eq:product_betas} directly follows from telescoping. 
  To prove this, we first note that $\calK_1(\mA,\vb) = \mathrm{span}(\vv_1)$ and thus 
  \begin{equation*}
    \beta_2 = \max_{\vw \in \calK_1(\mA,\vb), \|\vw\|=1} {\mathrm{dist}(\mA \vw, \calK_1(\mA,\vb))} = \mathrm{dist}(\mA \vv_1, \calK_1(\mA,\vb)) = \min_{\vu \in \mathcal{K}_1(\mA,\vb)} \|\mA \vv_1 -\vu\|.
  \end{equation*} 
  Moreover, since $\calK_j(\mA,\vb)$ is a linear subspace, we have $\mathrm{dist}(\lambda \mA \vw, \calK_j(\mA,\vb)) = \lambda \mathrm{dist}(\mA \vw, \calK_j(\mA,\vb))$ for any $\lambda > 0$. Thus, we can equivalently write \eqref{eq:beta_and_sigma} as 
  \begin{equation}\label{eq:beta_unnormalized}
    \beta_{j+1} =\max_{\vw \in \calK_j(\mA,\vb)}\frac{\mathrm{dist}(\mA \vw, \calK_j(\mA,\vb))}{\|\vw\|}.
  \end{equation}
  Note that for any $\vw \in \calK_j(\mA,\vb)$, it can be decomposed as $\vw = c_0 \mA^{j-1}\vv_1 - \vu$, where $c_0 \in \mathbb{R}$ is some {scalar} and $\vu \in \calK_{j-1}(\mA,\vb)$. If $c_0 = 0$, then $\vw \in \calK_{j-1}(\mA,\vb)$ and further we have $\mA\vw \in \calK_{j}(\mA,\vb)$, which implies that $\mathrm{dist}(\mA \vw, \calK^{(j)}) = 0$. Hence, for the maximization problem in \eqref{eq:beta_unnormalized},  without loss of generality, we can assume that $\vw = \mA^{j-1}\vv_1 - \vu$ with $\vu \in \calK_{j-1}(\mA,\vb)$ and rewrite \eqref{eq:beta_unnormalized} as  
  \begin{equation}\label{eq:beta_ratio}
    \beta_{j+1} = \max_{\vu \in \calK_{j-1}(\mA,\vb)}\frac{\mathrm{dist}(\mA^{j}\vv_1 - \mA\vu, \calK_j(\mA,\vb))}{\|\mA^{j-1}\vv_1 - \vu\|} = \max_{\vu \in \calK_{j-1}(\mA,\vb)}\frac{\mathrm{dist}(\mA^{j}\vv_1, \calK_j(\mA,\vb))}{\|\mA^{j-1}\vv_1 - \vu\|},
  \end{equation}
  where we used the fact that $\mA \vu \in \calK_j(\mA,\vb)$ in the second equality. Moreover, the nominator in \eqref{eq:beta_ratio} can be written as $\mathrm{dist}(\mA^{j}\vv_1, \calK_j(\mA,\vb)) = \min_{\vu \in \mathcal{K}_j(\mA,\vb)} \|\mA^{j} \vv_1 -\vu\|$, while the denominator is equal to $\min_{\vu \in \mathcal{K}_{j-1}(\mA,\vb)} \|\mA^{j-1} \vv_1 -\vu\|$. This completes the proof.   
\end{proof}

As a corollary of item (ii) in Lemma~\ref{lem:beta_product}, the quantity $\rho^{(m)}(\mA,\vb)$ defined in \eqref{eq:def_rho} can also be written as $\rho^{(m)}(\mA,\vb) = \left(\prod_{j=1}^m \beta_{j+1}\right)^{\frac{1}{m}}$. 
This observation will be used in the following lemma.  
\begin{lemma}\label{lem:bessel}
Assume that $\mA \succeq 0$. 
Then, for any $\vs\in \mathbb{R}^d$, we have 
  \begin{equation}\label{eq:bound_on_min}
      \min_{j \in \{1,2,\dots, m\}} \; \left\{(\mP^{(j)}\vs)^\top \mA \mP^{(j)}\vs -\vs^\top \mA \vs \right\} \leq  \frac{2\rho^{(m)}(\mA,\vb)}{m}\|\vs\|^2.
  \end{equation}
\end{lemma}
\begin{proof}
  To begin with, fix $j\in \{1,2,\dots,m\}$. By using $\vs = \mP^{(j)}\vs + (\mI-\mP^{(j)})\vs$, we can write 
  \begin{equation*}
      \vs^\top \mA \vs =  (\mP^{(j)}\vs)^\top \mA \mP^{(j)}\vs + 2 \left((\mI-\mP^{(j)})\vs\right)^{\top} \mA \mP^{(j)}\vs + \left((\mI-\mP^{(j)})\vs\right)^\top \mA (\mI-\mP^{(j)})\vs . 
  \end{equation*}
  Since $\mA \succeq 0$, we have $\left((\mI-\mP^{(j)})\vs\right)^\top \mA (\mI-\mP^{(j)})\vs  \geq 0$, which implies that 
  \begin{equation}\label{eq:crossover_bound}
       (\mP^{(j)}\vs)^\top \mA \mP^{(j)}\vs - \vs^\top \mA \vs  \leq  -  2 \left((\mI-\mP^{(j)})\vs\right)^{\top} \mA \mP^{(j)}\vs.  
  \end{equation}
  To bound the right-hand side of \eqref{eq:crossover_bound}, we first note that 
  $\left((\mI-\mP^{(j)})\vs\right)^{\top} \mA \mP^{(j)}\vs = \vs^\top (\mI-\mP^{(j)})\mA \mP^{(j)}\vs = \beta_{j+1} (\vv_{j+1}^\top \vs) (\vv_j^\top \vs)$, where we used  item (i) in Lemma~\ref{lem:beta_product} to obtain the last equality. Hence, \eqref{eq:crossover_bound} becomes 
  \begin{equation*}
    (\mP^{(j)}\vs)^\top \mA \mP^{(j)}\vs - \vs^\top \mA \vs  \leq -2 \beta_{j+1}(\vv_{j+1}^\top \vs) (\vv_j^\top \vs) \leq 2 \beta_{j+1} |\vv_j^\top \vs| |\vv_{j+1}^\top \vs|. 
  \end{equation*}
  Since the inequality above holds for any $j$, by taking the minimum over $j=1,2,\dots,m$, it further leads to 
  \begin{equation}\label{eq:min_both_sides}
    \min_{j \in \{1,2,\dots, m\}} \; \left\{(\mP^{(j)}\vs)^\top \mA \mP^{(j)}\vs -\vs^\top \mA \vs \right\} \leq  2 \min_{j \in \{1,2,\dots, m\}} \beta_{j+1} |\vv_j^\top \vs| |\vv_{j+1}^\top \vs|. 
  \end{equation}
  Since the minimum is upper bounded by the geometric mean, we also have 
  \begin{equation}\label{eq:min_geometric}
      \min_{j \in \{1,2,\dots, m\}} \beta_{j+1} |\vv_j^\top \vs| |\vv_{j+1}^\top \vs| \leq \Biggl(\prod_{j=1}^m \beta_{j+1} |\vv_j^\top \vs| |\vv_{j+1}^\top \vs|\Biggr)^{\frac{1}{m}} \leq \Biggl(\prod_{j=1}^m \beta_{j+1}\Biggr)^{\frac{1}{m}} \Biggl(\prod_{j=1}^m |\vv_j^\top \vs| |\vv_{j+1}^\top \vs|\Biggr)^{\frac{1}{m}}. 
  \end{equation}
  Recall that we have $\rho^{(m)}(\mA,\vb) = \left(\prod_{j=1}^m \beta_{j+1}\right)^{\frac{1}{m}}$. 
  Furthermore, note that $\{\vv_j\}_{j=1}^m$ are orthonormal and by Bessel's inequality we have 
  \begin{equation*}
      \sum_{j=1}^{m+1} |\vv_j^\top \vs|^2 \leq  \|\vs\|^2. 
  \end{equation*}
  Using the inequality of arithmetic and geometric means, we obtain 
  \begin{equation}\label{eq:AM_GM}
      \Biggl(\prod_{j=1}^m |\vv_j^\top \vs | |\vv_{j+1}^\top \vs |\Biggr)^{\frac{1}{m}} \!\!\leq \frac{1}{m}\sum_{i=1}^m |\vv_j^\top \vs | |\vv_{j+1}^\top \vs | \leq \frac{1}{m}\sum_{i=1}^m \frac{1}{2} (|\vv_j^\top \vs | ^2 + |\vv_{j+1}^\top \vs|^2 ) 
      \leq \frac{1}{m} \sum_{j=1}^{m+1} | \vv_j^\top \vs |^2 \leq \frac{1}{m} \|\vs\|^2. 
  \end{equation}
  Finally \eqref{eq:bound_on_min} follows from \eqref{eq:min_both_sides}, \eqref{eq:min_geometric} and \eqref{eq:AM_GM}.  
  This completes the proof. 
\end{proof}
\section{Proofs of The Main Theorems}\label{appen:main}

To begin with, recall that the columns of $\mV_k$ form an orthogonal basis for the Krylov subspace $\calK_m(\mH_k,\vg_k)$. Hence, for any $\vs \in \calK_m(\mH_k,\vg_k)$, there exists $\vz \in \mathbb{R}^m$ such that $\vs = \mV_k \vz$ and $\|\vz\| = \|\vs\|$.  Thus, by substituting $\mV_k \vz$ for $\vs$ in \eqref{eq:general_subspace}, the update rule of Algorithm~\ref{alg:subspace_cubic} can be equivalently written as 
\begin{align}
  \vs_{k} & := \argmin_{\vs \in \calK_m(\mH_k,\vg_k)} \Bigl\{ \vg_k^\top \vs + \frac{1}{2}\vs^\top \mH_k \vs + \frac{M}{6}\|\vs\|^3 \Bigr\} \label{eq:Krylov_cubic_newton} \\ 
  \vx_{k+1} & = \vx_k + \vs_k. \label{eq:Krylov_cubic_newton_2}  
\end{align} 
This alternative form of Algorithm~\ref{alg:subspace_cubic} will be useful in the subsequent proofs.  
\subsection{Proof of Lemma~\ref{lem:intermediate_step}}\label{appen:intermediate_step}

As the first step of proving Lemma~\ref{lem:intermediate_step}, we first present the following lemma, which shows a similar inequality as in \eqref{eq:key_inequality} but only holds for vector $\vs$ in the subspace $\mathcal{K}_m(\mH_k, \vg_k)$.  
\begin{lemma}\label{lem:one_step_krylov}
  For any $\vs \in \mathcal{K}_m(\mH_k, \vg_k)$, we have 
  \begin{equation*}
      f(\vx_{k+1}) \leq f(\vx_k) + \vg_k^\top \vs + \frac{1}{2}\vs^\top \mH_k\vs  + \frac{M}{6}\| \vs\|^3.
  \end{equation*}
\end{lemma}
\begin{proof}
  Since $M \geq L_2$ and $\vs_k = \vx_{k+1}-\vx_k$, 
  by applying Proposition~\ref{prop:taylor} with $\vx'=\vx_k$ and $\vx = \vx_{k+1}$ we have 
  \begin{equation*}
      f(\vx_{k+1}) \leq f(\vx_k) + \vg_k^\top \vs_k + \frac{1}{2}\vs_k^\top \mH_k\vs_k  + \frac{M}{6}\| \vs_k\|^3.
  \end{equation*}
  According to \eqref{eq:Krylov_cubic_newton}, $\vs_{k}$ is chosen as the minimizer of the cubic subproblem over the subspace $\mathcal{K}_m(\mH_k, \vg_k)$, which means that 
  $\vg_k^\top \vs_k + \frac{1}{2}\vs_k^\top \mH_k \vs_k + \frac{M}{6}\|\vs_k\|^3 \leq \vg_k^\top \vs + \frac{1}{2}\vs^\top \mH_k \vs + \frac{M}{6}\|\vs\|^3 $ for any $\vs \in \mathcal{K}_m(\mH_k, \vg_k)$. Hence, the result immediately follows. 
\end{proof}

Now we are ready to prove Lemma~\ref{lem:intermediate_step}. 
\begin{proof}[Proof of Lemma~\ref{lem:intermediate_step}] 
  Recall that $\mP_k^{(j)} = \mV_k^{(j)} \mV_k^{(j)\top}$ is the orthogonal projection matrix of the subspace $\mathcal{K}_j(\mH_k,\vg_k)$.
  For any given $\vs \in \mathbb{R}^d$, define $\vs^{(j)} = \mP_k^{(j)} \vs$ for $j=1,2,\dots,m$. Since we have $\vs^{(j)} \in \mathcal{K}_j(\mH_k, \vg_k) \subset \mathcal{K}_m(\mH_k, \vg_k)$ for any $j\leq m$, 
  by applying Lemma~\ref{lem:one_step_krylov} with $\vs = \vs^{(j)}$, we get 
  \begin{align*}
    f(\vx_{k+1}) &\leq f(\vx_k) +  \vg_k^\top \vs^{(j)} + \frac{1}{2}(\vs^{(j)})^\top \mH_k \vs^{(j)} + \frac{M}{6}\| \vs^{(j)}\|^3 \\
    &= f(\vx_k) + \vg_k^\top \mP_k^{(j)}\vs + \frac{1}{2}(\mP_k^{(j)} \vs )^\top\mH_k \mP_k^{(j)} \vs + \frac{M}{6}\| \mP_k^{(j)} \vs\|^3.
\end{align*}
  Moreover, since $\vg_k \in \calK_{j}(\mH_k,\vg_k)$ for all $j\geq 0$, we have $ \mP_k^{(j)}\vg_k = \vg_k$. Also, since $\mP_k^{(j)}$ is an orthogonal projection matrix, we have $\|\mP_k^{(j)} \vs\| \leq \|\vs\|$ for any $\vs\in \mathbb{R}^d$. 
  Thus, we get 
  \begin{align}
    f(\vx_{k+1}) &\leq f(\vx_k) +  \vg_k^\top \vs + \frac{1}{2}(\mP_k^{(j)} \vs )^\top\mH_k \mP_k^{(j)} \vs + \frac{M}{6}\| \vs\|^3 \nonumber \\
    &=  f(\vx_k) +  \vg_k^\top \vs + \frac{1}{2}\vs^\top \mH_k \vs + \frac{M}{6}\| \vs\|^3 +\frac{1}{2}\left((\mP_k^{(j)} \vs )^\top\mH_k \mP_k^{(j)} \vs - \vs^\top \mH_k \vs \right). \label{eq:intermediate}
  \end{align}
  Since \eqref{eq:intermediate} holds for any $j\in \{1,2,\dots,m\}$, we obtain the desired result by taking the minimum of the right-hand side over $j=1,\dots,m$.  
\end{proof}

\subsection{Proof of Lemma~\texorpdfstring{\ref{lem:bessel_main}}{3}}
Since $f$ is convex by Assumption~\ref{assum:convex}, we have $\mH_k \succeq 0$. Thus, Lemma~\ref{lem:bessel_main} immediately follows from Lemma~\ref{lem:bessel} by setting $\mA = \mH_k$ and $\vb = \vg_k$.  

\subsection{Proof of Theorem~\ref{thm:main}}\label{appen:main_theorem}
Our starting point is the inequality in \eqref{eq:key_inequality_krylov}. By choosing $\vs = 0$, we immediately obtain that $f(\vx_{k+1}) \leq f(\vx_k)$, which implies $f(\vx_k) \leq f(\vx_0)$ for all $k\geq 0$. Thus, $\vx_{k}$ is in the level-set $\{\vx\in \mathbb{R}^d: f(\vx) \leq f(\vx_0)\}$ and hence $\|\vx_k-\vx^*\| \leq D$ according to \eqref{eq:def_D}. Moreover,
by Proposition~\ref{prop:taylor}, it holds that $
    |f(\vx_k +\vs)- f(\vx_k) - \vg_k^\top \vs - \frac{1}{2}  \vs^\top \mH_k \vs| \leq \frac{L_2}{6}\|\vs\|^3$. 
Thus, by denoting $\rho^{(m)}(\mH_k,\vg_k)$ by $\rho^{(m)}_k$, together with \eqref{eq:key_inequality_krylov} we can get 
\begin{equation}\label{eq:proximal_type_bound}
    f(\vx_{k+1}) \leq f(\vx_k+ \vs) + \frac{\rho^{(m)}_k}{m}\|\vs\|^2 + \frac{L_2 + M}{6}\|\vs\|^3, \quad \forall \vs \in \mathbb{R}^d.  
\end{equation}
Now define two auxiliary sequences $\{a_k\}_{k\geq 0}$ and $\{A_k\}_{k\geq 0}$ by $a_k = k^2$, $A_0 = 0$ and $A_k = A_{k-1} + a_k$. By choosing $\vs = \frac{a_{k+1}}{A_{k+1}}(\vx^*-\vx_k)$ in \eqref{eq:proximal_type_bound}, we get 
\begin{align}
    f(\vx_{k+1}) &\leq f\left(\frac{A_k}{A_{k+1}}\vx_k + \frac{a_{k+1}}{A_{k+1}}\vx^*\right) + \frac{\rho^{(m)}_k}{m} \frac{a_{k+1}^2}{A_{k+1}^2}\|\vx^*-\vx_k\|^2 + \frac{L_2 + M}{6} \frac{a_{k+1}^3}{A_{k+1}^3}\|\vx^*-\vx_k\|^3 \\
    &\leq \frac{A_k}{A_{k+1}}f(\vx_k) + \frac{a_{k+1}}{A_{k+1}} f(\vx^*) + \frac{\rho^{(m)}_k}{m} \frac{a_{k+1}^2}{A_{k+1}^2}\|\vx^*-\vx_k\|^2 + \frac{L_2 + M}{6} \frac{a_{k+1}^3}{A_{k+1}^3}\|\vx^*-\vx_k\|^3,
\end{align}
where we used Jensen's inequality in the last inequality. By multiplying both sides by $A_{k+1}$ and {using the upper bound $\|\vx_k - \vx^*\| \leq D$}, we obtain 
\begin{equation*}
    A_{k+1}\left(f(\vx_{k+1}) - f(\vx^*)\right) \leq A_k(f(\vx_k)-f(\vx^*)) + \frac{\rho^{(m)}_k}{m} \cdot \frac{a_{k+1}^2}{A_{k+1}} D^2 + \frac{L_2 + M}{6} \cdot\frac{a_{k+1}^3}{A_{k+1}^2}D^3. 
\end{equation*}
Note that $A_{k+1} = \sum_{j=1}^{k+1} a_j = \sum_{j=1}^{k+1} j^2 \geq (k+1)^3/3$ and $a_{k+1} = (k+1)^2$. Hence, we further have 
\begin{equation*}
    A_{k+1}\left(f(\vx_{k+1}) - f(\vx^*)\right) \leq A_k(f(\vx_k)-f(\vx^*)) + \frac{3\rho^{(m)}_k}{m} (k+1) D^2 + \frac{3}{2}(L_2+M)D^3. 
\end{equation*}
By unrolling the above inequality, we obtain 
\begin{equation*}
    A_k(f(\vx_k)-f(\vx^*)) \leq \frac{3D^2}{m} \sum_{j=0}^{k-1} (j+1)\rho^{(m)}_j + \frac{3}{2}(L_2+M) D^3 k.
\end{equation*}
Finally, by using $\rho^{(m)}_j \leq \rho^{(m)}_{\max}$ and $A_k \geq k^3/3$, this implies that
\begin{equation*}
    f(\vx_k)-f(\vx^*) \leq \frac{9\rho^{(m)}_{\max}D^2}{2m}\left(\frac{1}{k}+\frac{1}{k^2}\right) + \frac{9(L_2+M)D^3}{2k^2}.
\end{equation*}

\subsection{Proof of Theorem~\ref{thm:strongly-convex}}\label{appen:strongly-convex}
 To simplify the notation, we define  $T^{(m)}_M:\;\mathbb{R}^d \rightarrow \mathbb{R}^d$ as the operator that maps the current iterate $\vx_k$ to the next iterate $\vx_{k+1}$ in Algorithm~\ref{alg:subspace_cubic}. Formally, from \eqref{eq:Krylov_cubic_newton} and \eqref{eq:Krylov_cubic_newton_2} we have  
\begin{equation*}
    T^{(m)}_M(\vx) =  \vx + \argmin_{\vs \in \calK_m(\nabla^2 f(\vx),\nabla f(\vx))} \Bigl\{ \langle \nabla f(\vx), \vs \rangle + \frac{1}{2}\vs^\top  \nabla^2 f(\vx) \vs + \frac{M}{6}\|\vs\|^3 \Bigr\}.
\end{equation*}
Thus, Algorithm~\ref{alg:subspace_cubic} can be equivalently written as $\vx_{k+1} = T^{(m)}_M(\vx_k)$ for $k\geq 0$. 

To prove Theorem~\ref{thm:strongly-convex}, we artificially divide the iterates of Algorithm~\ref{alg:subspace_cubic} into different stages and use $\vx_{i,t}$ to denote the $t$-th iterate in the $i$-th stage. Also, the $i$-th stage has length $k_i$ which we will specify later. Formally, in the first stage, we set $\vx_{1,0} = \vx_0$ and let $\vx_{1,t+1} = T_M^{(m)}(\vx_{1,t})$ for $0\leq t \leq {k_1-1}$. Subsequently, in the $i$-th stage ($i\geq 2$), we set the initial iterate $\vx_{i,0}$ as the last iterate of the previous stage $\vx_{i-1,k_{i-1}}$ and let $\vx_{i,t+1} = T_M^{(m)}(\vx_{i,t})$ for $0\leq t \leq {k_i-1}$.
Note that we have not altered the algorithm and the above procedure is exactly equivalent to Algorithm~\ref{alg:subspace_cubic}; The different stages are introduced solely for the purpose of analysis.   

We choose the length $k_i$ of the $i$-th stage such that the suboptimality gap halves after every stage, that is, $f(\vx_{i+1,0})-f(\vx^*) \leq \frac{1}{2}(f(\vx_{i,0})-f(\vx^*))$ for $i\geq 0$. In particular, we claim that it is sufficient to choose 
\begin{equation*}
  k_i \!=\!\! \left\lceil \frac{72 \rho^{(m)}_{\max}}{m \mu} + \frac{6 \cdot 2^{\frac{1}{4}}(L_2+M)^{\frac{1}{2}} \left( {f(\vx_{i,0}) - f(\vx^*)}\right)^{\frac{1}{4}}}{\mu^{\frac{3}{4}}} \right\rceil \!\leq\! \frac{72 \rho^{(m)}_{\max}}{m \mu} + \frac{8(L_2+M)^{\frac{1}{2}} \left( {f(\vx_{i,0}) - f(\vx^*)}\right)^{\frac{1}{4}}}{\mu^{\frac{3}{4}}} +1. \!
\end{equation*}

Let $D_i = \sup \{\|\vx-\vx^*\|:\; \vx\in \reals^d, \;f(\vx) \leq f(\vx_{i,0})\}$.  
Note that for any $\vx$ satisfying $f(\vx) \leq f(\vx_{i,0})$, by using strong convexity we can bound $ \frac{\mu}{2}\|\vx-\vx^*\|^2 \leq f(\vx)-f(\vx^*) \leq f(\vx_{i,0})-f(\vx^*)$. This implies that  $\|\vx-\vx^*\| \leq \sqrt{\frac{2(f(\vx_{i,0})-f(\vx^*))}{\mu}}$ and thus we have $D_i \leq  \sqrt{\frac{2(f(\vx_{i,0})-f(\vx^*))}{\mu}}$. Hence, by applying Theorem~\ref{thm:main}, we can bound that 
\begin{align*}
  f(\vx_{i,k_i})-f(\vx^*) &\leq \frac{9 \rho^{(m)}_{\max{}} D_i^2}{2m} \left(\frac{1}{k_i}+\frac{1}{k_i^2}\right) + \frac{9(L_2+M)D_i^3}{2k_i^2} \\
   &\leq \frac{9\rho^{(m)}_{\max}}{m \mu}\left( f(\vx_{i,0}) - f(\vx^*) \right) \left(\frac{1}{k_i}+\frac{1}{k_i^2}\right) + \frac{9\sqrt{2}(L_2+M)}{k_i^2} \left( \frac{f(\vx_{k,0}) - f(\vx^*)}{\mu} \right)^{\frac{3}{2}}.
\end{align*}
Since we choose $k_i$ such that $k_i \geq \frac{72 \rho^{(m)}_{\max}}{m \mu}$, we have  
\begin{equation*}
  \frac{9\rho^{(m)}_{\max}}{m \mu}\left( f(\vx_{i,0}) - f(\vx^*) \right) \left(\frac{1}{k_i}+\frac{1}{k_i^2}\right) \leq  \frac{18\rho^{(m)}_{\max}}{m \mu k_i}\left( f(\vx_{i,0}) - f(\vx^*) \right) \leq \frac{1}{4}\left( f(\vx_{i,0}) - f(\vx^*) \right).
\end{equation*}
Moreover, since $k_i \geq 6 \cdot 2^{\frac{1}{4}}\frac{(L_2+M)^{\frac{1}{2}} \left( {f(\vx_{i,0}) - f(\vx^*)}\right)^{\frac{1}{4}}}{\mu^{\frac{3}{4}}}$, we also have 
\begin{equation*}
  \frac{9\sqrt{2}(L_2+M)}{k_i^2} \left( \frac{f(\vx_{i,0}) - f(\vx^*)}{\mu} \right)^{\frac{3}{2}} \leq \frac{1}{4}\left( f(\vx_{i,0}) - f(\vx^*) \right).
\end{equation*}
By combining the two, we conclude that $f(\vx_{i+1,0})-f(\vx^*) = f(\vx_{i,k_i})-f(\vx^*) \leq \frac{1}{2}\left( f(\vx_{i,0}) - f(\vx^*) \right)$. 
By induction, we have $ f(\vx_{i,0})-f(\vx^*) \leq (f(\vx_0)-f(\vx^*))/2^{i-1} = {\delta_0}/2^{i-1}$. Hence, after at most $i^* = \lceil \log_2(\frac{\delta_0}{\epsilon})+1 \rceil$ stages, we have $f(\vx_{i^*,0}) \leq \epsilon$. Furthermore, the total number of iterations can be bounded by 
\begin{align*}
  \sum_{i=1}^{i^*-1} k_i &\leq \frac{72 \rho^{(m)}_{\max}}{m \mu} \left(\log_2\left(\frac{\delta_0}{\epsilon}\right)+1\right) + \log_2\left(\frac{\delta_0}{\epsilon}\right)+1 + \sum_{i=1}^{\infty}\frac{8(L_2+M)^{\frac{1}{2}} \left( {f(\vx_{i,0}) - f(\vx^*)}\right)^{\frac{1}{4}}}{\mu^{\frac{3}{4}}} \\ 
  &\leq \frac{72 \rho^{(m)}_{\max}}{m \mu} \left(\log_2\left(\frac{\delta_0}{\epsilon}\right)+1\right) + \log_2\left(\frac{\delta_0}{\epsilon}\right)+1 + \frac{8(L_2+M)^{\frac{1}{2}} \left( {f(\vx_{0}) - f(\vx^*)}\right)^{\frac{1}{4}}}{(1-2^{-\frac{1}{4}})\mu^{\frac{3}{4}}}. 
\end{align*}

\section{Upper Bounds on \texorpdfstring{$\rho^{(m)}$}{rho}} \label{appen:rho_m}

In this section, we first present the proof of Lemma~\ref{lem:matrix_polynomial}, which relates $\rho^{(m)}(\mA,\vb)$ to matrix polynomials. Then in the subsequent sections, we will use this lemma to derive different bounds on $\rho^{(m)}(\mA,\vb)$ by choosing different polynomials. 
As it will play an important role in our proof, we first briefly recap the definition of Chebyshev polynomials and present some useful properties \citep[Section 4.4]{saad2011numerical}. 

The Chebyshev polynomial of the first kind of degree $k$ can be defined by 
\begin{equation*}
  T_k(x) = \cos(k \cos^{-1}x).
\end{equation*}
Equivalently, it can also be defined via the following recurrence relation: 
\begin{equation*}
  T_0(x) = 1, \; T_1(x) = x, \; T_{k+1}(x) = 2x T_k(x) - T_{k-1}(x) \quad \forall k \geq 1.
\end{equation*}
For convenience, we list some useful properties of the Chebyshev polynomials in the following proposition.  
\begin{proposition}\label{prop:chebyshev}
  The following statements hold true. 
  \begin{enumerate}[(i),leftmargin=1cm]
    \item $|T_k(x)| \leq 1$ for all $x\in [-1,1]$ and $k \geq 0$.
    \item For $k\geq 1$, the leading coefficient of $T_k(x)$ is $2^{k-1}$. 
  \end{enumerate}
\end{proposition}

\subsection{Proof of Lemma~\texorpdfstring{\ref{lem:matrix_polynomial}}{4}}

By items (i) and (iii) in Proposition~\ref{lem:beta_product}, we can write  
\begin{equation*}
  \rho^{(m)}(\mA,\vb) = \left(\prod_{j=1}^m \beta_{j+1}\right)^{\frac{1}{m}} = \min_{\vu \in \mathcal{K}_m(\mA,\vb)} \|\mA^{m} \vv_1 -\vu\|^{\frac{1}{m}},
\end{equation*}
where we recall that $\vv_1 = \vb/\|\vb\|$. 
Moreover, for any $\vu \in \calK_m(\mA,\vb)$, by definition, there exist $c_0,\dots,c_{m-1} \in \reals$ such that 
$$\vu = \sum_{j=0}^{m-1} c_j {\mA^j\vb} = \left(\sum_{j=0}^{m-1} c_j \mA^j\right){\vb} = \left(-\sum_{j=0}^{m-1} \tilde{c}_j \mA^j\right) \frac{\vb}{\|\vb\|},$$
where we let $\tilde{c}_j = -c_j \|\vb\|$.  
Hence, we have 
$$\mA^{m}\vv_1 -\vu  = \left(\mA^{m} + \sum_{j=0}^{m-1} \tilde{c}_j \mA^j\right)\frac{\vb}{\|\vb\|} = \frac{p(\mA)\vb}{\|\vb\|},$$ 
where $p$ is a monic polynomial of degree $m$ as defined in Lemma~\ref{lem:matrix_polynomial}. 
Thus, minimizing $\vu$ over the subspace $\calK_m(\mA,\vb)$ is equivalent to minimizing $p$ over the set $\mathcal{M}_m$,  which leads to   
\begin{equation*}
    \rho^{(m)}(\mA,\vb) = \min_{p \in \mathcal{M}_m } \left\|p(\mA) \frac{\vb}{\|\vb\|} \right\|^{\frac{1}{m}} .
\end{equation*}
This completes the proof. 
\subsection{Proof of Lemma~\texorpdfstring{\ref{lem:L1_bound}}{2}}\label{appen:chebyshev_1}
Assume that $\mA$ has $r$ distinct eigenvalues $\lambda_1 > \lambda_2 > \dots > \lambda_r$, where $r\leq d$. As discussed in Section~\ref{subsec:hessian_structure}, for any monic polynomial $\hat{p} \in \mathcal{M}_m$, by using Lemma~\ref{lem:matrix_polynomial} we can bound 
\begin{equation}\label{eq:polynomial_boun d_rho}
  \rho^{(m)}(\mA,\vb) \leq \min_{p \in \mathcal{M}_m } \left\|p(\mA) \right\|_{\op}^{\frac{1}{m}} \leq \min_{{p \in \mathcal{M}_m} }\max_{i\in\{1,2,\dots,r\}} \; |p(\lambda_i) |^{\frac{1}{m}} \leq \max_{i\in\{1,2,\dots,r\}} \; |\hat{p}(\lambda_i) |^{\frac{1}{m}}.
\end{equation} 
Since $0 \preceq \mA \preceq L_1 \mI$, we have $0\leq \lambda_i \leq L_1$ for all $i=1,\dots,r$. Now we will choose the polynomial $\hat{p}$ as 
\begin{equation*}
  \hat{p}(\lambda) =  2^{-m+1}\left(\frac{L_1}{2}\right)^m T_m\left(\frac{2\lambda}{L_1}-1\right) = 2\left(\frac{L_1}{4}\right)^m T_m\left(\frac{2\lambda}{L_1}-1\right). 
\end{equation*}
By item (ii) in Proposition~\ref{prop:chebyshev}, we can verify that $\hat{p}$ is indeed a monic polynomial of degree $m$. Moreover, for any $\lambda \in [0,L_1]$, we can obtain from item (i) in Proposition~\ref{prop:chebyshev} that $|\hat{p}(\lambda)| \leq 2\left({L_1}/{4}\right)^m$. hence, we conclude that $\rho^{(m)}(\mA,\vb) \leq 2^{1/m} L_1/4$. 
\subsection{Proof of Lemma~\texorpdfstring{\ref{lem:eigen_decay}}{5}}
\label{appen:eigen_decay}
Assume that $\mH$ has $r$ distinct eigenvalues $\lambda_1 > \lambda_2 > \dots > \lambda_r$, where $r\leq d$.
We follow similar arguments as in Section~\ref{appen:chebyshev_1} but choose the polynomial $\hat{p}$ differently.  Specifically, in the first case when $m <r$, we let 
\begin{equation*}
  \hat{p}(\lambda) = 
    \prod_{j=1}^{m} (\lambda - \lambda_j).
\end{equation*}
It is easy to verify that $\hat{p}$ is a monic polynomial of degree $m$ with its zeros at $\lambda_1,\dots,\lambda_m$. Thus, we immediately get $\hat{p}(\lambda_i) = 0$ for $i\in\{1,2,\dots,m\}$. On the other hand, for any $i\in\{m+1,m+2,\dots, r\}$, we have $0 \leq \lambda_i \leq \lambda_{m}$ and thus $\hat{p}(\lambda_i) \leq \prod_{j=1}^{m} |\lambda_i - \lambda_j| \leq \prod_{j=1}^{m} \lambda_j$. Hence, by \eqref{eq:polynomial_boun d_rho} we conclude that $\rho^{(m)}(\mH,\vg) \leq \left(\prod_{j=1}^{m} \lambda_j \right)^{\frac{1}{m}}$. 

In the second case when $m \geq r$, we can let 
\begin{equation*}
  \hat{p}(\lambda) = 
    (\lambda - \lambda_1)^{m-r}\prod_{j=1}^{r} (\lambda - \lambda_j).
\end{equation*} 
We can observe that $\hat{p}$ is a monic polynomial of degree $m$ and it vanishes at all $\lambda_i$ for $i=1,2,\dots,r$. Thus, we conclude that $\rho^{(m)}(\mH,\vg) = 0$. 

{Finally, we prove our claim in Remark~\ref{rem:maximal_Krylov}. We first make a simple observation that $\mH^m \vg \in \mathcal{K}_m(\mH,\vg)$ if and only if $m \geq r_0$, where we recall $r_0$ denotes the dimension of the maximal Krylov subspace. This is because, by definition, we have $\mathcal{K}_m(\mH,\vg) \subsetneq \mathcal{K}_{m+1}(\mH,\vg)$ for $m < r_0$ and $\mathcal{K}_m(\mH,\vg) = \mathcal{K}_{m+1}(\mH,\vg)$ for $m \geq r_0$. To complete the proof,  Lemma~\ref{lem:matrix_polynomial} implies that $\rho^{(m)}(\mH,\vg)=0$ if and only if there exists a polynomial $p(x) = x^m+\sum_{i=0}^{m-1} c_i x^i \in \mathcal{M}_m$ such that $p(\mH)\vg = 0$, i.e., $\mH^m \vg = -\sum_{i=0}^{m-1} c_i\mH^i \vg$. Moreover, this is exactly equivalent to the condition that $\mH^m\vg \in \mathcal{K}_m(\mH,\vg)$, and hence by our initial observation above we have $\rho^{(m)}(\mH,\vg)=0$ if and only if $m\geq r_0$.   }

\subsection{Proof of Lemma~\texorpdfstring{\ref{lem:cluster}}{6}}
Inspired by \citet{goujaud2022super}, we choose the polynomial $\hat{p}$ as   
\begin{equation*}
  \hat{p}(\lambda) = 2^{1-\frac{m}{2}}\left(\frac{\Delta(L_1-\Delta)}{2}\right)^{\frac{m}{2}} T_{m/2}(\sigma(\lambda)), \quad \text{where }\sigma(\lambda) = 1-\frac{2}{\Delta (L_1-\Delta)}\lambda (L_1-\lambda).
\end{equation*} 
We can verify that this is a monic polynomial of degree $m$. Moreover, for any $\lambda \in [0,\Delta] \cup [L_1-\Delta, L_1]$, it holds that $\sigma(\lambda) \in [-1,1]$. Hence, using item (i) in Proposition~\ref{prop:chebyshev}, for any $i=1,2,\dots,r$ we obtain that $|T_{m/2}(\sigma(\lambda_i))| \leq 1$,  which further implies that $|\hat{p}(\lambda)| \leq 2^{1-\frac{m}{2}}\left(\frac{\Delta(L_1-\Delta)}{2}\right)^{\frac{m}{2}}$. Thus, by \eqref{eq:polynomial_boun d_rho} we get $\rho^{(m)}(\mH,\vg) \leq 2^{1/m}\sqrt{\Delta(L_1-\Delta)}/2$.

\section{Experiment Details}

In the experiments, we focus on the logistic regression problem with LIBSVM datasets \citep{chang2011libsvm}. Specifically, consider a dataset $\{(\va_j,b_j)\}_{j=1}^{n}$ of $n$ points, where $\va_j \in \reals^d$ is the $j$-th feature vector and $b_j \in \{0,1\}$ is the $j$-th binary label. 
Then, the logistic loss function is given by 
$f(\vx)= \frac{1}{n}\sum_{j=1}^n \left((1-b_j) \va_j^\top \vx + \log(1+e^{- \va_j^\top \vx}) \right)$. We tested the CRN method \citep{nesterov2006cubic}, the SSCN method \citep{Hanzely2020}, and our proposed Krylov CRN method (Algorithm~\ref{alg:subspace_cubic}). 
In the following sections, we further describe their implementation details.

\subsection{Cubic Regularized Newton}

\begin{algorithm}[t!]
  \caption{Cubic regularized Newton with line search}\label{alg:cubic_LS}
  \begin{algorithmic}[1]
  \State \textbf{Input:}  Initial point $\vx_0 \in \mathbb{R}^d$, initial regularization parameter $R_0>0$, line search parameter $\beta \in (0,1)$
  \For{$k=0,1,\dots,$}
  \State Let $M_k$ be the smallest number in $\{R_k \beta^{-i}:i\geq 0\}$ such that 
  \begin{align*}
    \vs_{k} & = \argmin_{\vs \in \mathbb{R}^d} \Bigl\{ \vg_k^\top \vs + \frac{1}{2}\vs^\top \mH_k \vs + \frac{M_k}{6}\|\vs\|^3 \Bigr\}, \\ 
    \vx_{k+1} & = \vx_k + \vs_k, \\
     f(\vx_{k+1}) &\leq f(\vx_k) +  \vg_k^\top \vs  + \frac{1}{2}\vs^\top \mH_k \vs + \frac{M_k}{6}\|\vs\|^3. 
\end{align*}
\vspace{-1em}
  \State Set $R_{k+1} \leftarrow \beta M_k$
  \EndFor
  \end{algorithmic}
  \end{algorithm}

  \begin{algorithm}[t!]
    \caption{Stochastic Subspace Cubic Newton with line search}\label{alg:SSCN_LS}
    \begin{algorithmic}[1]
    \State \textbf{Input:}  Initial point $\vx_0 \in \mathbb{R}^d$, subspace dimension $m$, initial regularization parameter $R_0>0$, line search parameter $\beta \in (0,1)$
    \For{$k=0,1,\dots,$}
    \State Sample $m$ coordinates uniformly and randomly and let $I_k$ be the set of sampled indices
    \State Compute the subspace gradient $\tilde{\vg}_k$ and the subspace Hessian $\tilde{\mH}_k$
    \State  
    Let $M_k$ be the smallest number in $\{R_k \beta^{-i}:i\geq 0\}$ such that 
      \begin{align*}
          \vz_{k} &= \argmin_{\vz \in \mathbb{R}^m}\; \Bigl\{ \tilde{\vg}_k^\top \vz  + \frac{1}{2} \vz^\top \tilde{\mH}_k \vz  + \frac{M_k}{6}\|\vz\|^3 \Bigr\}, \\
          \vx_{k+1}[I_k] &= \vx_k[I_k] +  \vz_k, \\
          f(\vx_{k+1}) &\leq f(\vx_k) +  \vg_k^\top \vs  + \frac{1}{2}\vs^\top \mH_k \vs + \frac{M_k}{6}\|\vs\|^3.
      \end{align*}
      \vspace{-1em}
      \State Set $R_{k+1} \leftarrow \beta M_k$
    \EndFor
    \end{algorithmic}
    \end{algorithm}

    \begin{algorithm}[t!]
      \caption{Krylov cubic regularized Newton with line search}\label{alg:krylov_cubic_LS}
      \begin{algorithmic}[1]
      \State \textbf{Input:}  Initial point $\vx_0 \in \mathbb{R}^d$, subspace dimension $m$, initial regularization parameter $R_0>0$, line search parameter $\beta \in (0,1)$
      \For{$k=0,1,\dots,$}
      \State Set $(\mV_k, \tilde{\vg}_k, \tilde{\mH}_k) = \textsc{Lanczos}(\mH_k, \vg_k;m)$ %
      \State 
      Let $M_k$ be the smallest number in $\{R_k \beta^{-i}:i\geq 0\}$ such that 
      \begin{align*}
          \vz_{k} &= \argmin_{\vz \in \mathbb{R}^m}\; \Bigl\{ \tilde{\vg}_k^\top \vz  + \frac{1}{2} \vz^\top \tilde{\mH}_k \vz  + \frac{M_k}{6}\|\vz\|^3 \Bigr\}, \\
          \vx_{k+1} &= \vx_k + \mV_k \vz_k, \\
          f(\vx_{k+1}) &\leq f(\vx_k) +  \vg_k^\top \vs  + \frac{1}{2}\vs^\top \mH_k \vs + \frac{M_k}{6}\|\vs\|^3.
      \end{align*}
      \vspace{-1em}
      \State Set $R_{k+1} \leftarrow \beta M_k$
      \EndFor
      \end{algorithmic}
      \end{algorithm}
Recall that in the analysis of CRN, the regularization parameter $M$ should satisfy $M \geq L_2$. Since the Lipschitz constant $L_2$ is unknown in practice, we use a backtracking line search scheme to select $M_k$ at the $k$-th iteration, as described in Algorithm~\ref{alg:cubic_LS}. Specifically, at the $k$-th iteration, we iteratively increase the value of the regularization parameter $M_k$ until the key inequality in \eqref{eq:key_inequality} is satisfied. 

To solve the cubic subproblem in \eqref{eq:cubic_newton}, we follow the standard approach in \citet[Section 6.1]{cartis2011adaptive}. As we mentioned in Section~\ref{subsec:subspace_newton}, by using the first-order optimality condition, the cubic subproblem can be reformulated as the nonlinear equation: 
\begin{equation*}
  \lambda = \frac{M}{2}\|(\mH_k + \lambda \mI)^{-1} \vg_k\|.
\end{equation*} 
Thus, we define the univariate function $\phi(\lambda) = \lambda^2 - \frac{M^2}{4}\|(\mH_k + \lambda \mI)^{-1} \vg_k\|^2$ and use the Newton's method to find its unique root $\lambda^*$. Then, we can compute the solution by $\vs_k = -(\mH_k + \lambda^* \mI)^{-1} \vg_k$. Note that at each iteration of Newton's method, we need to solve a linear system of equations in the form of $(\mH_k+\lambda \mI)\vs = -\vg_k$ with a different $\lambda$. When the problem dimension $d$ is small (less than $500$ in our experiment), we can store the Hessian matrix $\mH_k$ and solve the linear system directly by computing the matrix inverse. On the other hand, when~$d$ is large, we need to use the conjugate gradient method to solve the linear system relying on Hessian-vector products. 
\subsection{Stochastic Subspace Cubic Newton}
We implemented the coordinate version of SSCN, where we sample $m$ coordinates uniformly and randomly at each iteration.  
We use a similar backtracking line search scheme in SSCN to select the regularization parameter $M_k$ as shown in Algorithm~\ref{alg:SSCN_LS}. Given an index set $I$, we use $\vx[I]$ to denote the subvector indexed by $I$. 
Moreover, to achieve its best performance, we follow the strategy in \citep[Section 7.1]{Hanzely2020} and store the residuals $\va_j^\top \vx_k$ for each data point $j=1,2,\dots,n$ and at each iteration $k$. As shown in \citet{Hanzely2020}, in this case the computational cost of the subspace gradient and the subspace Hessian is $\bigO(nm)$ and $\bigO(nm^2)$, respectively.  
\subsection{Krylov Cubic Regularized Newton}
To implement our proposed method, we also use a backtracking line search scheme to select the regularization parameter. The full algorithm is given in Algorithm~\ref{alg:krylov_cubic_LS}.